\documentclass{amsart}
\usepackage{amssymb}
\usepackage{amscd}
\usepackage{amsfonts}
\usepackage{latexsym}
\usepackage[all]{xy}
\usepackage{verbatim}

\newtheorem{theorem}{Theorem}[section]
\newtheorem{lemma}[theorem]{Lemma}
\newtheorem{proposition}[theorem]{Proposition}
\newtheorem{corollary}[theorem]{Corollary} 
\theoremstyle{definition}  

\newtheorem{example}[theorem]{Example}

\newtheorem{remark}[theorem]{Remark}

\newcommand{\Tr}{\text{Tr}}
\newcommand{\id}{\text{id}}

\newcommand{\Fun}{\text{Fun}}

\newcommand{\FPdim}{\text{FPdim}} 
\newcommand{\End}{\text{End}} 
\newcommand{\Irr}{\text{Irr}} 
\renewcommand{\Vec}{\text{Vec}}

\newcommand{\Hom}{\text{Hom}}

\newcommand{\Rep}{\text{Rep}}

\newcommand{\bF}{{\bar F}}

\newcommand{\C}{\mathcal{C}}
\newcommand{\D}{\mathcal{D}}

\newcommand{\Z}{\mathcal{Z}}

\newcommand{\M}{\mathcal{M}}
\newcommand{\A}{\mathcal{A}}

\newcommand{\N}{\mathcal{N}}

\newcommand{\V}{\mathcal{V}}
\renewcommand{\O}{\mathcal{O}}

\newcommand{\be}{\mathbf{1}}

\renewcommand{\be}{\mathbf{1}}

\newcommand{\BZ}{{\mathbb Z}}
\newcommand{\BC}{{\mathbb C}}
\newcommand{\BQ}{{\mathbb Q}}

\newcommand{\bt}{\boxtimes}

\newcommand{\ot}{\otimes}

\begin{document}

\title{Pivotal fusion categories of rank 3}

\author{Victor Ostrik}
\address{V.O.: Department of Mathematics,
University of Oregon, Eugene, OR 97403, USA}
\email{vostrik@math.uoregon.edu}
\address{D.N.: Department of Mathematics and Statistics,
University of New Hampshire,  Durham, NH 03824, USA}
\email{nikshych@math.unh.edu}

\begin{abstract} We classify all fusion categories of rank 3 that admit a pivotal structure over an
algebraically closed field of characteristic zero. Also in the Appendix (joint with D.~Nikshych) we
give some restrictions on Grothendieck rings of near-group categories.
\end{abstract}

\date{\today} 
\maketitle  

\section{Introduction}

Fusion categories introduced in \cite{ENO} are in a sense simplest tensor categories. The problem of classification of fusion categories while
interesting seems to be well out of reach at the moment. Thus it is natural to attempt a classification of fusion categories which are ``small''
in certain sense. One classical interpretation of ``small'' coming from the theory of subfactors is that the category in question contains an algebra object of small dimension; we refer the reader to \cite{JMS} for a nice survey of achievements in this direction.
Another possibility is that ``small'' categories have integral Frobenius-Perron dimension with small number of prime factors;
some results in this direction were obtained in \cite{ENOw}.

In this paper we follow approach initiated in \cite{O2} in which ``small'' categories have small {\em rank}.
We recall that rank of a fusion category is just a number of isomorphism classes of its simple objects.
The fusion categories of rank 1 are trivial and the fusion categories of rank 2 were classified in \cite{O2}. 
Under an additional assumption that the category in question is ribbon some further classification results
were obtained in \cite{O3}, \cite{RSW}, \cite{RS}. Also recently a general finiteness conjecture
on the number of modular tensor categories of fixed rank was proved, see \cite{BNRS}.

In this paper we give a classification of fusion categories of rank 3 that admit a pivotal structure
(this is relatively mild assumption: it is conjectured in \cite[Conjecture 2.8]{ENO} that any fusion
category admits a pivotal structure). Thus we prove the following result conjectured in \cite{O3}:

\begin{theorem} \label{uraa}
Let $\C$ be a fusion category of rank 3 admitting a pivotal structure over an
algebraically closed field of characteristic zero.
Then $\C$ is equivalent to one of the following:

{\em (i)} pointed category with underlying group $\BZ/3\BZ$;

{\em (ii)} even part of the category associated with quantum $sl_2$ at $7$th root of unity, see e.g.
\cite[Section 4.4]{O3};

{\em (iii)} Ising category, see e.g. \cite[Appendix B]{DGNO};

{\em (iv)} category of representations of the group $S_3$ or its twisted version, see \cite[Section 4.4]{O3};

{\em (v)} category associated with subfactor of type $E_6$ or its Galois conjugate, see e.g. \cite{HH}.
\end{theorem}

The proof of this result is significantly more difficult than the previous classification in rank 2 \cite{O2}. Thus we had to develop
some new theoretical tools. Most important are Theorem \ref{I1} which gives some information on the decomposition
of the induction of the unit object to the Drinfeld center of a fusion category and pseudo-unitary inequality
\eqref{psunit} which gives some non-trivial restrictions on the Grothendieck rings of fusion categories which admit
a spherical structure with positive values of dimensions. This inequality is strong enough to show that vast
majority of based rings of rank 3 are not categorifiable. It is also sometimes useful in some other cases,
see e.g. \cite{HL}.
However the computer experiments performed
by Nicolle Sandoval Gonzalez seem to indicate that this inequality is less efficient in higher ranks.

It is highly unsatisfactory that we had to assume the existence of pivotal structure for our results. 
Despite a significant effort we were not able to eliminate this assumption. However the existence
of fusion categories of rank 3  which do not admit such a structure seems to be very unlikely: 
we prove that 
such categories must have very big fusion coefficients, see Theorem \ref{bigfus}.

An interesting class of fusion categories is formed by {\em near-group categories} 
introduced in \cite{Sie}. By definition the near-group categories are ``small'' in the
following sense: the number of non-invertible simple objects in such categories is
precisely 1. 
The results of \cite{Iz} and \cite{EG} suggest that this class contains an infinite
family of fusion categories which seems to be unrelated to previously known constructions
using finite groups and quantum groups at roots of unity. In the Appendix \ref{A} written jointly
with Dmitry Nikshych we apply our
techniques to this case and obtain some new results on near-group categories. In particular
we show that  non-integral near-group categories must have commutative Grothendieck rings
and that integral near-group categories are weakly group theoretical in the sense of \cite{ENOw}.
We also note that Theorem \ref{near} gives a classification of near-group categories
with the group of invertible objects of order 2.

It is my great pleasure to express my deep gratitude to Dmitri Nikshych who made many 
useful suggestions for this paper. In particular he suggested to use results of Ng and 
Schauenburg in the proof of Theorem \ref{near} which allowed me to eliminate one 
particularly difficult case. Also Example \ref{xuex} is a result of our joint work. 
I am also happy to thank Nicolle Sandoval Gonzalez and Hannah Larson for many useful discussions.
Both authors of the Appendix are grateful to Masaki Izumi for sending his preprint \cite{Izz}.

\section{Preliminaries}
In this paper the base field $k$ is assumed to be algebraically closed of characteristic zero.

\subsection{Fusion categories and based rings} \label{prefus}
We recall from \cite{ENO} that a fusion category is a rigid tensor semisimple category over $k$
such that all Hom spaces are finite dimensional, number of isomorphism classed of simple objects is finite, and the unit object is simple. For a fusion category $\C$ we will denote by $\O(\C)$ the set
of isomorphism classes of simple objects in $\C$.

A {\em pivotal structure} on a fusion category $\C$ is an
isomorphism of tensor functors $\id \to \phantom{}^{**}$, see e.g. \cite[Definition 2.7]{ENO}.
In the presence of pivotal structure one defines traces of morphisms and dimensions $\dim(X)\in k$ of objects, see e.g. \cite[Section 2.2]{ENO}.
A pivotal structure (or the underlying fusion category) is {\em spherical} if $\dim(X)=\dim(X^*)$ for 
all objects $X\in \C$ (equivalently, all dimensions are totally real), see \cite[Section 2.2]{ENO}. 
For a fusion category $\C$ one defines its {\em global dimension} $\dim(\C)$, 
see e.g. \cite[Definition 2.2]{ENO}. For a spherical fusion category $\C$ the global dimension
is given by
$$\dim(\C)=\sum_{X\in \O(\C)}\dim(X)^2.$$
Assume that we fixed an embedding of the subfield of algebraic numbers in $k$ to $\BC$.
A {\em pseudo-unitary spherical structure} on a fusion category $\C$ is a spherical structure such 
that the images under this embedding of dimensions of all objects are positive real numbers,
see \cite[Section 8.4]{ENO}.
A fusion category is {\em pseudo-unitary} if it has a pseudo-unitary spherical structure.

The Grothendieck ring $K(\C)$ of a fusion category $\C$ is an example of {\em unital based ring} 
of finite rank (see e.g. \cite[Definition 2.1]{Od}). 
This means that $K(\C)$ is endowed with a basis $\{ b_i\}_{i\in I}$ 
over $\BZ$ (given by the classes of simple objects) such that the structure constants are nonnegative; moreover there exists an involution $i\mapsto i^*$ such that the expansion of 
$b_ib_{j}$ contains 1 (which is one of the basis elements) with coefficient $\delta_{ij^*}$.

For a based ring $K$ of finite rank its {\em categorification} is a fusion category $\C$ and an isomorphism of based rings $K\simeq K(\C)$; two categorifications are {\em equivalent} if there is a tensor equivalence $\C_1\to \C_2$ inducing the identity map on $K$. It is known that a given based ring of finite rank admits
only finitely many categorifications up to equivalence (this statement is known as ``Ocneanu rigidity'', see \cite[Theorem 2.28]{ENO}). We say that a based ring is {\em categorifiable} if it admits at least one categorification. The following statements are useful necessary conditions for a based ring $K$ to be categorifiable:

\begin{proposition}
{\em ({\bf Cyclotomic test}) (\cite[Corollary 8.53]{ENO})} Let $\C$ be a fusion category. Any irreducible representation of the ring $K(\C)$ is defined over some cyclotomic field. In particular for any homomorphism $\phi: K(\C)\to \BC$ we have $\phi([X])\in \BZ[\zeta]$ for some root of unity $\zeta$.\qed
\end{proposition}

We recall (see \cite{Od}) that an algebraic integer $\alpha$ is a {\em d-number} if 
$\frac{\sigma(\alpha)}{\alpha}$ is algebraic integer for any
automorphism $\sigma$ of $\bar \BQ$. Equivalently, if $x^n+a_1x^{n-1}+\cdots +a_n\in \BZ[x]$ is
the minimal polynomial of $\alpha$ then $\frac{a_i^n}{a_n^i}\in \BZ$ for $i=1, \ldots , n-1$.

\begin{proposition} \label{dnumt}
{\em ({\bf d-number test}) (\cite[Theorem 1.2]{Od})} Let $\C$ be a fusion category.
Let $R: K(\C)\to K(\C)$ be the operator of left multiplication by the element $\sum_ib_ib_{i^*}$.
Then the eigenvalues of $R$ are $d-$numbers. \qed
\end{proposition} 

\begin{remark} Proposition \ref{dnumt} is equivalent to \cite[Theorem 1.2]{Od} in view of 
Remark \ref{trRinv1} and \cite[Corollary 2.8]{Od}.
\end{remark}

\subsection{The induction functor $I : \C \to \Z(\C)$}\label{precenter}
 Let $\C$ be a spherical fusion category. It is known that its {\em Drinfeld center} $\Z(\C)$ is a {\em modular tensor category}, see \cite[Theorem 1.2]{Mu}, \cite[Theorem 2.15]{ENO}.
Since $\Z(\C)$ is semisimple the forgetful functor $F: \Z(\C)\to \C$ has a right adjoint $I: \C \to \Z(\C)$. We will use the following properties of this functor:

\begin{proposition} \label{FI}
{\em (\cite[Proposition 5.4]{ENO})} $F(I(Y))=\bigoplus_{X\in \O(\C)}X\ot Y\ot X^*$.
\end{proposition}

In particular we see that for any $Y\in \C$
\begin{equation} \label{dimI}
\dim(I(Y))=\dim(Y)\dim(\C).
\end{equation}

Recall that any modular tensor category is equipped with balance isomorphisms 
$\theta_X: X\to X$, see \cite[Section 2.2]{BK}.

\begin{theorem} \label{zero} {\em (cf. \cite[Remark 4.6]{NS})}
We have $\theta_{I(\be)}=\id_{I(\be)}$ and $\Tr(\theta_{I(\be)})=\dim(\C)$.
For any $\be \ne X\in \O(\C)$ we have $\Tr(\theta_{I(X)})=0$.
\end{theorem}

\begin{proof} The equality $\theta_{I(\be)}=\id_{I(\be)}$ follows from \cite[Corollary 2.7]{Oex}; then
$\Tr(\theta_{I(\be)})=\dim(I(\be))=\dim(\C)$ follows from \eqref{dimI}. 

Recall (see \cite[Lemma 3.5]{DMNO}) that $I(\be)$ is an {\em \'etale algebra} in $\Z(\C)$. Let
$\varepsilon : I(\be)\to \be$ be a unique map which composes to identity with the unit morphism.
Then the morphism $e: I(\be)\ot I(\be)\to I(\be)\xrightarrow{\varepsilon} \be$ is {\em non-degenerate}, see
\cite[Remark 3.4]{DMNO}. Thus there is a unique morphism $i: \be \to I(\be)\ot I(\be)$ such that
the pair $(e,i)$ satisfies the rigidity axiom, see \cite[Definition 1.11]{KiO}. For any $I(\be)-$module $M$
with the action morphism $a: I(\be)\ot M\to M$ consider the following composition (cf. \cite[Fig. 16]{KiO}):
$$\psi_M:M\xrightarrow{i\ot \id} I(\be)\ot I(\be)\ot M\xrightarrow{\id \ot a}I(\be)\ot 
M\xrightarrow{\id \ot \theta}I(\be)\ot M\xrightarrow{a} M.$$
Using \cite[Lemma 1.14]{KiO} one shows that $\Tr(\psi_M)=\dim(I(\be))\Tr(\theta_M)$.
On the other hand it is proved
in \cite[Lemma 4.3]{KiO} that for a simple $M$ we have $\psi_M=0$ unless $M$ is {\em dyslectic}. 

The result follows since the simple $I(\be)-$modules are precisely $I(X), X\in \O(\C)$
(see \cite[Lemma 3.5]{DMNO}), and  the only simple {\em dyslectic}
module over $I(\be)$ is $I(\be)$ itself, see \cite[Proposition 4.4]{DMNO}.
\end{proof}

\begin{example} Recall that for a modular category $\D$ its {\em Gauss sum} is
$$G(\D)=\sum_{A\in \O(\D)}\theta_A\dim(A)^2.$$
We can compute the Gauss sum of the category $\Z(\C)$ using Theorem \ref{zero}:
$$\sum_{A\in \O(\Z(\C))}\theta_A\dim(A)^2=\sum_{A\in \O(\Z(\C)), X\in \O(\C)}\theta_A\dim(A)
[F(A):X]\dim(X)=$$
$$\sum_{A\in \O(\Z(\C)), X\in \O(\C)}\theta_A\dim(A)
[I(X):A]\dim(X)=\sum_{X\in \O(\C)}\Tr(\theta_{I(X)})\dim(X)=\dim(\C)$$
Thus we recover the result of M\"uger \cite[Theorem 1.2]{Mu} stating that for any spherical $\C$ 
$$G(\Z(\C))=\dim(\C).$$
\end{example}

Ng and Schauenburg connected $\Tr(\theta_{I(X)}^i), i\in \BZ$ with {\em higher Frobenius-Schur 
indicators}, see \cite{NS}. In particular, $\Tr(\theta_{I(X)}^2)$ is related with the classical
Frobenius-Schur indicator. Thus we have

\begin{theorem} \label{FS}
{\em (\cite[Theorem 4.1]{NS})}
Assume that $X\in \O(\C)$ is self-dual. Then
$$\Tr(\theta_{I(X)}^2)=\pm \dim(\C). \;\;\; \square$$
\end{theorem}

Let $c_{A,B}: A\ot B\to B\ot A$ denote the braiding on the category $\Z(\C)$.
For $A,B\in \O(\Z(\C))$ let $s_{A,B}=\Tr(c_{B,A}\circ c_{A,B})$. We can consider
$s_{A,B}$ as a matrix of a linear operator on $K(\Z(\C))\otimes k$ with respect to
the basis consisting of classes $[A], A\in \O(\C)$. The following result was conjectured
by A.~Kitaev.

\begin{proposition} \label{I1eigen} {\em (cf. \cite[Theorem 4.1]{KiO})}
 The class $[I(\be)]\in K(\Z(\C)]$ is an eigenvector for $s=s_{A,B}$ with
eigenvalue $\dim(\C)$.
\end{proposition}

\begin{proof}
Recall that $c_{I(\be),A}\circ c_{A,I(\be)}=\theta_{A\ot I(\be)} (\theta_A \theta_{I(\be)})^{-1}$, see
e.g. \cite[(2.2.8)]{BK}. If $A$ is a simple object then $\theta_A$ is a scalar; 
also $\theta_{I(\be)}=\id_{I(\be)}$
by Theorem \ref{zero}. Thus 
$$\Tr(c_{I(\be),A}\circ c_{A,I(\be)})=\theta_A^{-1}\Tr(\theta_{A\ot I(\be)}).$$
Observe that $A\ot I(\be)$ is $I(\be)-$module, so it is a direct sum of modules $I(X), X\in \O(\C)$;
in view of Theorem \ref{zero} the trace above depends only on the multiplicity of $I(\be)$ in
$A\ot I(\be)$. This multiplicity is 
$$\dim_k\Hom_{I(\be)}(A\ot I(\be),I(\be))=\dim_k\Hom(A,I(\be))=[I(\be):A].$$ 
Notice that $[I(\be):A]\ne 0$ implies $\theta_A=1$ by Theorem \ref{zero}. Thus 
$$s\cdot [I(\be)]=\sum_{A\in \O(\Z(\C))}\Tr(c_{I(\be),A}\circ c_{A,I(\be)})[A]=$$
$$\sum_{A\in \O(\Z(\C))}[I(\be):A]\dim(\C)[A]=\dim(\C)[I(\be)].$$
\end{proof} 

\subsection{Formal codegrees and decomposition of $I(\be)$}
Let $K$ be a based ring of finite rank and let $E$ be an irreducible representation of $K$ 
over the field $k$. Then the element $\alpha_E=\sum_{i\in I}\Tr(b_i,E)b_{i^*}\in K\ot k$ is in the
center of $K\ot k$; it acts by zero on any irreducible representation of $K$ not isomorphic to $E$
and by a nonzero scalar $f_E\in k$ on $E$, see \cite[Lemma 2.3]{Od}. Let $\Irr(K)$ be the set of isomorphism classes
of irreducible representations of $K$. The scalars $f_E, E\in \Irr(K)$ are called {\em formal
codegrees} of the based ring $K$ or of the fusion category $\C$ with $K=K(\C)$.

\begin{example} \label{moduex}
Assume that $\D$ is a modular tensor category. It is well known that
for any $A\in \O(\D)$ the assignment $B\mapsto \frac{s_{A,B}}{\dim(A)}$ is a 1-dimensional
representation $\hat E_A$ of $K(\D)$; moreover the map $A\mapsto \hat E_A$ is a bijection
$\O(\D)\to \Irr(K(\D))$, see e.g. \cite[Theorem 3.1.12]{BK}. We have (see \cite[Section 3.3]{Od})
\begin{equation}
f_{\hat E_A}=\frac{\dim(\D)}{\dim(A)^2}.
\end{equation}
\end{example}

\begin{proposition} \label{trRinv}
For any based ring $K$ we have
$$\sum_{E\in \Irr(K)}\frac{\dim(E)}{f_E}=1.$$
\end{proposition}

\begin{proof} Use \cite[Proposition 19.2(c)]{Lu} with $a=1$.
\end{proof}

\begin{remark} \label{trRinv1}
It follows from \cite[Lemma 2.6]{Od} that the eigenvalues of the operator $R$ (see 
Proposition \ref{dnumt})
are $\dim(E)f_E$ appearing with multiplicity $\dim(E)^2$. Thus Proposition \ref{trRinv} states
that $\Tr(R^{-1})=1$.
\end{remark}

\begin{remark} \label{fctotp}
It was observed in \cite[proof of Lemma 8.49]{ENO} that the operator $R$
is positive definite. Thus Remark \ref{trRinv1} implies that $\dim(E)f_E$ is totally
positive for any $E\in \Irr(K(\C))$. Thanks to Proposition \ref{trRinv} this can be
improved to $f_E\ge \dim(E)$.
In particular we have $f_E\ge 1$ for any
representation $E$.
\end{remark}

Let $\C$ be a spherical fusion category and let $E\in \Irr(K(\C))$. Then $K(\Z(\C))$ acts on $E$
via the forgetful homomorphism $K(\Z(\C))\to K(\C)$; by Schur's lemma this action factorizes via
a character of $K(\Z(\C))$. Let $\hat E\in \Irr(K(\Z(\C)))$ be the corresponding 1-dimensional
representation of $K(\Z(\C))$. Recall (see Section \ref{precenter}) that $\Z(\C)$ is a modular tensor category
with $\dim(\Z(\C))=\dim(\C)^2$. Thus by Example \ref{moduex} there exists a unique $A_E\in \O(\Z(\C))$
such that $\hat E=\hat E_{A_E}$; in addition
\begin{equation} \label{squareofw}
f_{\hat E}=\frac{\dim(\C)^2}{\dim(A_E)^2}.
\end{equation}

\begin{theorem} \label{I1}
For a spherical fusion category $\C$
the assignment $E\mapsto A_E$ is an embedding $\Irr(K(\C))\subset \O(\Z(\C))$. Its image consists
of $A\in \O(\Z(\C))$ with $[I(\be):A]\ne 0$. Moreover, we have
\begin{equation} \label{multipl}
[I(\be):A_E]=\dim(E),\;\; \mbox{and}\;\; \dim{A_E}=\frac{\dim(\C)}{f_E}.
\end{equation}
\end{theorem}

\begin{proof} It is known that $K(\Z(\C))\ot k$ maps surjectively on the center of $K(\C)\ot k$, see
\cite[Lemma 8.49]{ENO}; the first statement follows. Also \cite[Lemma 3.1]{Od} states that
$f_{\hat E}=f_{E}^2$, so \eqref{squareofw} implies $f_E=\pm \frac{\dim(\C)}{\dim(A_E)}$.

In view of Proposition \ref{I1eigen} and Example \ref{moduex} we have 
$$[I(\be)]=\sum_{A\in \O(\Z(\C))}[I(\be):A]\frac{\dim(A)}{\dim(\C)}\alpha_{\hat E_A}.$$
Thus $[I(\be)]$ acts on $E\in \Irr(K(\C))$ by the scalar 
$$[I(\be):A_E]\frac{\dim(A_E)}{\dim(\C)}f_{\hat E_{A_E}}=[I(\be):A_E]\frac{\dim(\C)}{\dim(A_E)},$$
see \eqref{squareofw}. On the other hand by Proposition \ref{FI} we have $[F(I(\be)]=R$, so by
\cite[Lemma 2.6]{Od} it acts on $E$ by the scalar $\dim(E)f_E$. Hence
$$\dim(E)f_E=[I(\be):A_E]\frac{\dim(\C)}{\dim(A_E)}.$$
Combining this with $f_E=\pm \frac{\dim(\C)}{\dim(A_E)}$ we get \eqref{multipl}.

It is known (see \cite[Proposition 5.6]{ENO}) that
$\sum_{A\in \O(\Z(\C))}[I(\be):A]^2$ equals the rank of $\C$. 
We already proved that each object $A_E$ appears in $I(\be)$ with multiplicity $\dim(E)$. Thus
$$\sum_{E\in \Irr(K(\C))}[I(\be):A_E]^2=\sum_{E\in \Irr(K(\C))}\dim(E)^2=\dim(K(\C)\ot k)=
\sum_{A\in \O(\Z(\C))}[I(\be):A]^2.$$
It follows that any object $A\in \O(\Z(\C))$ with $[I(\be):A]\ne 0$ is of the form $A_E$ for some
$E\in \Irr(K(\C))$. Theorem is proved.
\end{proof}

\begin{corollary} \label{fcdivis}
{\em (cf \cite[Proposition 8.22]{ENO})} For a not necessary spherical fusion category $\C$
the formal codegrees $f_E$ divide $\dim(\C)$.
\end{corollary}

\begin{proof} In a case of spherical $\C$ the result is immediate from Theorem \ref{I1} since 
$\dim(A_E)$ is an algebraic integer. For a general fusion category $\C$ one considers its
{\em sphericalization} $\tilde \C$, see \cite[Remark 3.1]{ENO} or Section \ref{sphn} below. This is spherical
category of global dimension $2\dim(\C)$; moreover $2f_E$ is a formal codegree of $\tilde \C$
for any $E\in \Irr(K(\C))$. The result follows.
\end{proof}

Assume again that $\C$ is spherical. Let $L(\C)$ be the {\em dimension field} of $\C$, that is the 
field generated over $\BQ$ by the dimensions of simple objects of $\C$.  

\begin{corollary} \label{dfield}
The formal codegrees of $\C$ lie in the field $L(\C)$.
\end{corollary}

\begin{proof} For any $A\in \O(\Z(\C))$ we have $\dim(A)=\dim(F(A))\in K$. Since $\dim(\C)\in K$ the result
follows from \eqref{multipl}.
\end{proof}

\begin{corollary}\label{sumofm}
There is an isomorphism of associative algebras {\em 
$$K(\C)\ot k\simeq \End(I(\be)).$$}
\end{corollary}

\begin{proof} Both algebras are isomorphic to direct sum of matrix algebras of sizes $\dim(E)$ indexed by $E\in \Irr(K(\C))$.
\end{proof}

We don't know whether there is a canonical choice of the isomorphism in Corollary \ref{sumofm}. 

\begin{example} Assume that $\C$ is dual to a modular tensor category $\D$. 
Then $\Z(\C)=\D \boxtimes \D^{rev}$ (see \cite[Theorem 7.10]{Mu})
and $[I(\be)]\in K(\Z(\C))=K(\D)\ot K(\D)$ is known as {\em modular invariant} associated 
with $\C$ (or rather with module category $\M$ over $\D$ such that $\C=\D^*_{\M}$), 
see \cite[Definition 5.5]{BEK}, \cite[Section 5.3]{FSR}. 
In this case our results are well known: 
Theorem \ref{zero} says that the modular invariant commutes with $T-$matrix 
(cf \cite[Theorem 5.7]{BEK},
\cite[Theorem 5.1]{FSR}), Proposition \ref{I1eigen} says that modular
invariant commutes with $S-$matrix (cf \cite[Theorem 5.7]{BEK}, \cite[Theorem 5.1]{FSR}), and 
Corollary \ref{sumofm} expresses $K(\C)\ot k$ in terms of the modular invariant
(cf \cite[Theorem 6.8]{BEK}, \cite[Theorem O]{FRSb}).
\end{example}

\begin{example} \label{xuex}
(joint with D.~Nikshych) It follows immediately from Corollary \ref{sumofm} that $K(\C)$ is commutative if and 
only if $I(\be)$ is multiplicity free. Let us consider minimal \'etale subalgebras of $I(\be)$. Such subalgebras intersect trivially,
so no nontrivial simple summand of $I(\be)$ appears in two such subalgebras. Thus the number of such subalgebras is at
most $rk(\C)-1$. In view of \cite[Theorem 4.10]{DMNO} this translates into the following statement: the number of maximal fusion subcategories of $\C$ is at most $rk(\C)-1$. It was conjectured in 
\cite[Conjecture 1.2]{GX} that this statement holds for any fusion category $\C$
(possibly with non-commutative $K(\C)$). We don't know whether this statement fails for a unital based ring $K$ of finite rank.
\end{example}

\begin{remark} \label{funMN}
 Let $\C=\oplus_{i,j}\C_{ij}$ be an indecomposable multi-fusion category (see \cite[Section 2.4]{ENO}) 
with decomposable unit object $\be=\oplus_i\be_i$. Assume that $\Z(\C_1)$ 
is spherical and $\theta_{I(\be)}=\id_{I(\be)}$. 
 The proof of Theorem \ref{I1} and Corollary \ref{sumofm} extends with trivial changes to this setting. 
Moreover, Theorem \ref{I1} applied to fusion category $\C_{ii}$ implies that the dimension of $[\be_i]E$
equals $[I(\be_i):A_E]$. It follows that we can choose an isomorphism in 
Corollary \ref{sumofm} sending $[\be_i]\in K(\C_1)$ to the projection of $I(\be)$ onto $I(\be_i)$. 
This implies that $(K(\C_{ii})\ot k, K(\C_{jj})\ot k)-$bimodule $K(\C_{ij})\ot k$ can be identified 
with $(\End(I(\be_i)), \End(I(\be_j)))-$bimodule $\Hom(I(\be_i),I(\be_j))$ where $K(\C_{ii})\ot k$ and 
$K(\C_{jj})\ot k$ are identified with $\End(I(\be_i))$ and $\End((\be_j))$ as in Corollary \ref{sumofm}.
In particular, the number of simple objects in $\C_{ij}$ equals $\dim \Hom(I(\be_i),I(\be_j))$. 

For example this applies as follows: let $\D$ be a spherical fusion category and let
$\M, \N$ be indecomposable semisimple module categories over $\D$. Let $M,N$ be Lagrangian subalgebras in $\Z(\D)$ associated with $\M$ and $\N$ in \cite[Section 4.2]{DMNO} 
(so $M=I_{\M}(\be_\M)$ where $I_\M$ is the adjoint of the canonical
functor $\Z(\D)\to \D^*_\M$ to the dual of $\D$ with respect to $\M$, and similarly for $N$). Assume that 
$\theta_M=\id_M$ and $\theta_N=\id_N$. Then applying the above statement to $\C=\D^*_{\M\oplus \N}$ we get
the following: $(K(\D^*_\M)\ot k,K(\D^*_\N)\ot k)-$bimodule $K(\Fun_\D(\M,\N))\ot k$
is isomorphic to $(\End(M),\End(N))-$bimodule $\Hom(M,N)$ where $K(\D^*_\M)\ot k$ is identified 
with $\End(M)$ as 
in Corollary \ref{sumofm} and similarly for $K(\D^*_\N)$. In particular the number of simple objects in 
$\Fun_\D(\M,\N)$ equals $\dim \Hom(M,N)$. This statement can be considered as a generalization
of the theory of quantum Manin triples from \cite[Section 4.4]{DMNO}. 
In the case of modular category $\D$ this result was obtained in \cite[Theorem 6.12]{BEK} and
\cite[Theorem 5.18]{FSR}.

Note also that the condition $\theta_M=\id_M$ (and similalrly for $N$) is satisfied automatically when the dimensions
in the category $\Z(\D)$ are positive. In this case the category $\D^*_\M$ is pseudo-unitary. Indeed, it is known that
$\Z(\D_\M^*)\simeq \Z(\D)$ (see \cite[Remark 3.18]{MuI}, \cite[Corollary 2.6]{Oorbi}), 
so the result follows from \cite[Theorem 2.15 and Proposition 8.12]{ENO}.
Thus the spherical structure on $\Z(\D)=\Z(\D^*_\M)$ is the same as the one obtained from a spherical structure on 
$\D^*_\M$ (defined in \cite[Proposition 8.23]{ENO}) and one deduces $\theta_M=\id_M$ from Theorem \ref{zero}.
On the other hand the condition $\theta_M=\id_M$ is not always satisfied: it fails for $\D=\Vec_{\BZ/2\BZ}$ 
(see \cite[Example 1 in Section 2]{ENO}) with non-standard
spherical structure (so the dimension of the nontrivial object is -1) and module category $\M$ with one simple object.
Thus it would be interesting to investigate whether the conclusions above hold true in the case when $\theta_M=\id_M$
fails (this is the case in the example above since we can replace the non-standard spherical structure by the standard one).
\end{remark}

\begin{example} Let $\C$ be a spherical fusion category and let $I(\be)\in \Z(\C)$ be as above. 
It was asked by A.~Kitaev whether there exists an \'etale algebra $\tilde I$ such that $\tilde I=I(\be)$
as an object of $\Z(\C)$ but the \'etale algebras $\tilde I$ and $I(\be)$ are not isomorphic. 
In view of Remark \ref{funMN} this translates into the following question: is there a module
category $\M$ over $\C$ such that $K(\C)-$module $K(\M)\ot k$ is isomorphic to $K(\C)\ot k$? 
Such examples do exist: for example let $\V$ be $\BZ/2\BZ-$graded fusion category constructed in 
\cite[Theorem A.5.1]{CMS}; then $\V_1$ considered as a module category over $\C=\V_0$
satisfies the condition above. Thus the question above has a positive answer.
\end{example}

\subsection{Pseudo-unitary inequality}
Let us assume that the category $\C$ is {\em pseudo-unitary}, see \cite[Section 8.4]{ENO}. Thus the category
$\C$ has a spherical structure such that the dimensions of objects are nonnegative.

\begin{theorem} Let $\C$ be a pseudo-unitary fusion category. Then the formal codegrees
of $\C$ satisfy
\begin{equation} \label{psunit}
\sum_{E\in \Irr(K(\C))}\frac1{f_E^2}\le \frac12\left(1+\frac1{\dim(\C)}\right).
\end{equation}
\end{theorem}

\begin{remark} It follows from Remark \ref{trRinv1} that the LHS of \eqref{psunit} equals $\Tr(R^{-2})$.
Since $\dim(\C)$ is the largest eigenvalue of $R$, the inequality \eqref{psunit} can be stated purely
in terms of operator $R$. 
\end{remark}

\begin{proof} Let $X\in \O(\C)\setminus \{ \be \}$. By Theorem \ref{zero}
$$\sum_{A\in \O(\Z(\C))}[I(X):A]\theta_A\dim(A)=\Tr(\theta_{I(X)})=0,$$
and $\theta_A=1$ for any $A$ appearing in $I(\be)$. Thus
$$\sum_{\substack{A\in \O(\Z(\C))\\ [I(\be):A]\ne 0}}[I(X):A]\dim(A)=
-\sum_{\substack{A\in \O(\Z(\C))\\ [I(\be):A]=0}}[I(X):A]\theta_A\dim(A).$$
Since all dimensions are positive real numbers we have
$$\sum_{\substack{A\in \O(\Z(\C))\\ [I(\be):A]\ne 0}}[I(X):A]\dim(A)=
\left|\sum_{\substack{A\in \O(\Z(\C))\\ [I(\be):A]\ne 0}}[I(X):A]\dim(A)\right|$$
$$=\left|\sum_{\substack{A\in \O(\Z(\C))\\ [I(\be):A]=0}}[I(X):A]\theta_A\dim(A)\right|\le 
\sum_{\substack{A\in \O(\Z(\C))\\ [I(\be):A]=0}}[I(X):A]\dim(A).$$
Therefore
$$\sum_{\substack{A\in \O(\Z(\C))\\ [I(\be):A]\ne 0}}[I(X):A]\dim(A)\le \frac12
\sum_{A\in \O(\Z(\C))}[I(X):A]\dim(A)=\frac12\dim(I(X)).$$
Recall that $\dim(I(X))=\dim(X)\dim(\C)$, see \eqref{dimI}. Hence
multiplying the inequality above by $\dim(X)$ and summing over $X\in \O(\C)\setminus \{ \be \}$ we get
$$\sum_{\be \ne X\in \O(\C)}\sum_{\substack{A\in \O(\Z(\C))\\ [I(\be):A]\ne 0}}[I(X):A]\dim(X)\dim(A)\le \sum_{\be \ne X\in \O(\C)}\frac12\dim(X)^2\dim(\C).$$
Observe that 
$$\sum_{\be \ne X\in \O(\C)}[I(X):A]\dim(X)=\sum_{\be \ne X\in \O(\C)}[F(A):X]\dim(X)=
\dim(A)-[F(A):\be],$$
and $$\sum_{\be \ne X\in \O(\C)}\dim(X)^2=\dim(\C)-1.$$ Therefore we get
\begin{equation}\label{pochti}
\sum_{\substack{A\in \O(\Z(\C))\\ [F(A):\be]\ne 0}}(\dim(A)-[F(A):\be])\dim(A)\le 
\frac12(\dim(\C)-1)\dim(\C).
\end{equation}
Using \eqref{dimI} we compute
$$\sum_{\substack{A\in \O(\Z(\C))\\ [I(\be):A]\ne 0}}[F(A):\be]\dim(A)=
\sum_{\substack{A\in \O(\Z(\C))\\ [I(\be):A]\ne 0}}[I(\be):A]\dim(A)=\dim(I(\be))=\dim(\C).$$
Thus \eqref{pochti} is equivalent to
$$\left(\sum_{\substack{A\in \O(\Z(\C))\\ [I(\be):A]\ne 0}}\dim(A)^2\right)-\dim(\C)\le 
\frac12(\dim(\C)-1)\dim(\C),$$
or 
$$\sum_{\substack{A\in \O(\Z(\C))\\ [I(\be):A]\ne 0}}\dim(A)^2\le 
\frac12(\dim(\C)+1)\dim(\C).$$
By Theorem \ref{I1} this inequality is equivalent to \eqref{psunit}.
\end{proof} 

\section{Categorifiable simple based rings of rank 3}
\subsection{Based rings of rank 3} 
Let $k,l,m,n$ be integers satisfying the equality
\begin{equation} \label{klmn}
k^2+l^2=kn+lm+1.
\end{equation}
It is easy to check that a ring $K(k,l,m,n)$ with basis $1, X, Y$ and multiplication
\begin{equation}\label{frules}
X^2=1+mX+kY,\; Y^2=1+lX+nY,\; XY=YX=kX+lY
\end{equation}
is associative. In a case when the numbers $k,l,m,n$ are nonnegative the ring $K(k,l,m,n)$ is
a based ring.

Another based ring of rank 3 is the group algebra of the group $\BZ/3\BZ$ with basis given by the group elements which we denote as $K(\BZ/3\BZ)$.

We have the following 

\begin{proposition} \label{based3}
{\em (cf \cite[\S 3.1]{O3})} Any unital based ring of rank 3 is isomorphic to either $K(k,l,m,n)$ or 
$K(\BZ/3\BZ)$.
\end{proposition}

\begin{proof} In the case when any element of the based ring is self-dual \eqref{frules} is a general
form of fusion rules and \eqref{klmn} is equivalent to the associativity of multiplication \eqref{frules}.

Assume that not every element of the based ring is self-dual. Then we have a basis $1,X,Y$ with
$X^*=Y$ whence $\FPdim(X)=\FPdim(Y)$. The element $X^2$ is an integral linear combination
of $X$ and $Y$, hence $\FPdim(X)\in \BZ$. The element $XY$ is 1 plus an integral linear combination of $X$ and $Y$, so $\FPdim(X)$ is a divisor of 1. Thus $\FPdim(X)=\FPdim(Y)=1$
and the Proposition is proved.
\end{proof}

\begin{remark} \label{bas}
(i) Notice that the interchange $X\leftrightarrow Y$ induces a based ring isomorphism
$K(k,l,m,n)\simeq K(l,k,n,m)$. Thus we will often assume that $l\le k$.

(ii) The diophantine equation \eqref{klmn} has infinitely many solutions. Clearly, for any such solution $k$ and $l$ are coprime. The number of (nonnegative) solutions
with $0<l\le k\le x$ is about $Cx^2\ln(x)$ for a constant $C\approx \frac13$ (notice that there are
infinitely many solutions with $k=1$, $l=0$).
\end{remark} 

The main result of this section is the following

\begin{theorem} \label{simple3}
Let $\C$ be a pivotal fusion category of rank 3. Then $K(\C)$ is isomorphic to one of the following:

$K(1,0,m,0)$ for some nonnegative integer $m$;

$K(1,1,1,0)$;

$K(\BZ/3\BZ)$. 
\end{theorem}

By Proposition \ref{based3} we can and will assume that $K(\C)=K(k,l,m,n)$ in the proof of Theorem 
\ref{simple3}. 

{\em Proof of Theorem \ref{simple3}.} 
Combine Lemma \ref{pivot}, Propositions \ref{generic3} and
\ref{small3} with Lemmas \ref{b50} and \ref{2121} below. $\square$

\subsection{Notations} \label{notations}
 The following notations will be used throughout the paper. We consider an element
$1+X^2+Y^2=3+(l+m)X+(k+n)Y\in K(k,l,m,n)$ and the operator $R: K(k,l,m,n)\to K(k,l,m,n)$ 
of left multiplication by this element. We have explicitly
\begin{equation} \label{matrix}
R=\left(\begin{array}{ccc}
3&l+m&k+n\\
l+m&3+(l+m)m+(k+n)k&(l+m)k+(k+n)l\\
k+n&(l+m)k+(k+n)l&3+(l+m)l+(k+n)n
\end{array}\right)
\end{equation}

In what follows $p(t)=\det(t-R)$ denote the characteristic polynomial of $R$. The roots of the polynomial $p(t)$ are real positive (since the matrix above is symmetric positive definite); we will
denote the roots by $f_1, f_2, f_3$ assuming that $f_1\le f_2\le f_3$. It is clear that the numbers
$f_1, f_2, f_3$ are precisely the formal codegrees of $K(k,l,m,n)$. Thus Proposition \ref{trRinv} implies that
$\frac1{f_1}+\frac1{f_2}+\frac1{f_3}=1$, or, equivalently, $p(t)=t^3-bt^2+ct-c$ for some positive
integers $b=\Tr(R)$ and $c=\det(R)$. Using \eqref{matrix} we have
\begin{equation} \label{beq}
b=(k+n)^2+(l+m)^2+9
\end{equation}
The formula for $c$ is more complicated, see \eqref{magic} below.

\subsection{Pivotal categories of rank 3 are pseudo-unitary}
\begin{lemma} \label{pivot}
Let $\C$ be a pivotal fusion category of rank 3. Then some Galois conjugate of $\C$ is pseudo-unitary.
\end{lemma}

\begin{proof} Since $K(\C)$ is commutative semisimple ring there are 3 distinct homomorphisms
$\phi_i: K(\C)\to \BC, i=1,2,3$. One of this homomorphisms is the categorical dimension and 
another one is the Frobenius-Perron dimension. We would like to show that these two homomorphisms are in the same orbit under the action of the Galois group. For a sake of contradiction assume that this is not the case. Then the orbit of one of them has size 1, that is
one of these homomorphisms lands into $\BZ$. If the Frobenius-Perron dimension lands in $\BZ$
then the category $\C$ is integral and we are done by \cite[Proposition 8.24]{ENO}. If the categorical dimension lands in $\BZ$ the category $\C$ is pseudo-unitary by \cite[Lemma A.1]{RS}.
\end{proof}

\subsection{The generic case}
The pseudo-unitary inequality \eqref{psunit} applied to $K(k,l,m,n)$ says:

$$1-\frac{4b}c\le \frac1{f_3}.$$

If $c-4b>0$ this is equivalent to
\begin{equation} \label{psu3}
f_3\le \frac{c}{c-4b}.
\end{equation}

\begin{proposition} \label{generic3}
Assume that $c-4b>4$ and $b\ge 50$. Then there is no pseudo-unitary
fusion category $\C$ such that $K(\C)=K(k,l,m,n)$.
\end{proposition}

\begin{proof} 
The Vieta theorem implies that $f_3\ge \frac13b>9$. Under the assumptions of the Proposition
$c>4b$, so \eqref{psu3} applies. We get
$$\frac{c}{c-4b}>9\Leftrightarrow c<\frac92b.$$

\begin{lemma} Under the assumptions of the Proposition, we have
$$1<f_1<2<f_2<b-5<f_3<b.$$
\end{lemma}

\begin{proof} We have $p(1)=1-b<0$, $p(2)=8+c-4b>0$ and $p(b)=bc-c>0$. Also
$$p(b-5)=(b-5)((b-5)^2-b(b-5)+c)-c=(b-5)(c-\frac92b+\frac12(50-b))-c<0.$$
The Lemma follows.
\end{proof}

By the assumptions of the Proposition we have $c-4b\ge 5$. Thus
$$\frac{c}5\ge \frac{c}{c-4b}\ge f_3>b-5,$$
whence
$$c>5b-25\Leftrightarrow (c-\frac92b)+\frac12(50-b)>0.$$
We get a contradiction with inequalities $b\ge 50$ and $c<\frac92b$ which proves the Proposition.
\end{proof}
\subsection{Small cases}
In this section we consider the cases not covered by Proposition \ref{generic3}. Thus either $b<50$
or $c-4b\le 4$, or both.

\subsubsection{The case $b<50$}
\begin{lemma} \label{b50}
Let $\C$ be a fusion category such that $K(\C)=K(k,l,m,n)$ with $b<50$.
Then $K(\C)$ is isomorphic to one of the following: $K(1,0,m,0)$ with $m\le 6$, $K(1,1,1,0)$,
$K(2,1,2,1)$.
\end{lemma}
 
\begin{proof} Using \eqref{beq} and \eqref{klmn} we get
$$b=3k^2+3l^2+n^2+m^2+7.$$
Assuming $k\ge l>0$ we get from inequality $b<50$ that $k\le 3$. There are 11 based rings 
$K(k,l,m,n)$ with $3\ge k\ge l>0$ but out of those only $K(1,1,1,0)$ and $K(2,1,2,1)$ pass both
the cyclotomic and d-number tests.

If $l=0$ then $k=1$ since $k$ and $l$ are coprime. By \eqref{klmn} this implies that $n=0$ 
and $b=10+m^2$. The inequality $b<50$ gives $m\le 6$ and we are done.
\end{proof}

\subsubsection{The case $c-4b\le 4$} 
\begin{proposition} \label{small3}
Let $\C$ be a fusion category such that $K(\C)=K(k,l,m,n)$ with $c-4b\le 4$.
Then $K(\C)$ is isomorphic to one of the following: $K(1,0,m,0)$ with $m\in \BZ_{\ge 0}$, 
$K(1,1,1,0)$, $K(2,1,2,1)$.
\end{proposition}

\begin{proof}
We have the following formula verified by a direct calculation using \eqref{matrix}:
\begin{equation}\label{magic}
p(2)=8+c-4b=(kl-mn)^2+(k^2+kn-l^2-ml)^2-1.
\end{equation}
Hence the inequality $c-4b\le 4$ is equivalent to
\begin{equation} \label{magicineq}
(kl-mn)^2+(k^2+kn-l^2-ml)^2\le 13.
\end{equation}

Observe that in view of \eqref{klmn} the number $k^2+kn-l^2-lm=2(k^2-lm)-1$ is odd. By Remark
\ref{bas} (i) we can assume that this number is positive. Then \eqref{magicineq} implies that
$k^2-lm=1$ or $k^2-lm=2$. Now we consider these cases separately.

{\em Case $k^2-lm=1$.} In this case $l^2-kn=0$ by \eqref{klmn}. Since $k$ and $l$ are coprime
this implies $k=1, l^2=n, lm=0$. In the case $l=0$ we get the based ring $K(1,0,m,0)$. Otherwise,
$m=0$ and \eqref{magicineq} says
$$l^2+1\le 13\Leftrightarrow l\le 3.$$
The rings $K(1,3,0,9)$ (with $b=118, c=473$) and $K(1,2,0,4)$ (with $b=38, c=148$) do not pass neither cyclotomic nor d-number tests. Thus the only possibility is $K(1,1,0,1)\simeq K(1,1,1,0)$ and the conclusion of the Proposition holds in this case.

{\em Case $k^2-lm=2$.} In this case $l^2-kn=-1$ whence $m=\frac{k^2-2}l$ and $n=\frac{l^2+1}k$.
We compute $$kl-mn=\frac{2l^2-k^2+2}{kl}.$$ The inequality \eqref{magicineq} says
$(kl-mn)^2+9\le 13$, so we have the following subcases:

{\em Subcase $kl-mn=\pm 2$.} Equivalently, $2l^2-k^2+2=\pm 2kl$ implying $3l^2+2=(k\pm l)^2$
and we get a contradiction considering this equality modulo 3.

{\em Subcase $kl-mn=\pm 1$.} Equivalently, $2l^2-k^2+2=\pm kl$ implying $9l^2+8=(2k\pm l)^2$
and again we get a contradiction considering this equality modulo 3.

{\em Subcase $kl-mn=0$.} Equivalently, $2l^2-k^2+2=0$. This diophantine equation has infinitely many solutions. We compute $m=2l, n=\frac{k}2, b=27\frac{l^2+1}2$. Equation \eqref{magic} gives $c=4b$. The polynomial $p(t)=t^3-bt^2+4bt-4b$ satisfies $p(2)=8$. Hence if $\gamma$ is an integer root of $p(t)$ then $2-\gamma$ is a divisor of 8 and one verifies easily that $p(t)$ is irreducible for $b>27$. Observe that $b=27\frac{l^2+1}2$ is odd, so 
$\frac{b^3}c=\frac{b^2}4\not \in \BZ$. Thus for $b>27$ the d-number test is failed (one can also show that the cyclotomic test fails as well). Finally, $27\frac{l^2+1}2=b\le 27$ implies $l=1$ and
we get the based ring $K(2,1,2,1)$, so the conclusion of the Proposition holds.

The Proposition is proved.
\end{proof}

\begin{remark} The crucial role in the proof above is played by the identity \eqref{magic}. Thus it seems to be of great interest to find its counterparts for based rings of higher ranks.
\end{remark}

\subsection{The ring $K(2,1,2,1)$}
\begin{lemma} \label{2121}
There is no pseudo-unitary fusion category $\C$ with $K(\C)=K(2,1,2,1)$.
\end{lemma}

\begin{proof} We have $p(t)=t^3-27t^2+108t-108$ in this case, so $f_1=12-6\sqrt{3}, f_2=3, f_3=12+6\sqrt{3}$. Using Theorem \ref{I1} we get that $I(\be)\in \Z(\C)$ decomposes into simple objects as $\be \oplus A \oplus B$ where $F(A)=\be \oplus X\oplus Y$ and 
$F(B)=\be \oplus 2X\oplus 2Y$. Using Proposition \ref{FI} we see that $I(Y)$ decomposes as 
$A+2B+S$ where $F(A+2B)=3+5X+5Y$ and $S$ is a (non simple) object with $F(S)=4X+4Y$. 
This implies that
$\Tr(\theta_{I(Y)})\ne 0$ and we have a contradiction with Theorem \ref{zero}.
\end{proof}

\begin{remark} We will show later (see Theorem \ref{quadra}) that a categorification of $K(2,1,2,1)$ admits a spherical (and hence pivotal) structure.
Thus by the Lemma above and Lemma \ref{pivot} the ring $K(2,1,2,1)$ is not categorifiable.
\end{remark}

\section{Non-simple based rings of rank 3} 
\subsection{}The main result of this section is the following

\begin{theorem} \label{near}
 Assume that the based ring $K(1,0,m,0)$ is categorifiable. Then $m\le 2$. 
\end{theorem}

The multiplication in the based ring $K(1,0,m,0)$ is given by
\begin{equation}
Y^2=1,\; XY=YX=X,\; X^2=1+mX+Y
\end{equation}

Thus, any category with $K(\C)=K(1,0,m,0)$ is an example of {\em near-group category} associated
with finite group $\BZ/2\BZ$, see e.g. \cite[Definition 1.1]{Th}. In particular, such category or its Galois
conjugate is pseudo-unitary by \cite[Theorem 1.3]{Th}.

In order to prove Theorem \ref{near} we will assume for the sake of contradiction that $m>2$ and
$\C$ is a pseudo-unitary fusion category with $K(\C)=K(1,0,m,0)$. We will endow $\C$ with the pseudo-unitary spherical structure. We set $d=\dim (X)=\FPdim(X)$.

\begin{lemma} We have $d=\frac{m}2+\sqrt{\frac{m^2}4+2}$. In particular $d\not \in \BZ$ for
$m\ge 2$. $\square$
\end{lemma}


\subsection{Simple objects in $\Z(\C)$} One of the formal codegrees of the based ring 
$K(1,0,m,0)$ is $f_1=2$. By Theorem \ref{I1} the object $I(\be)\in \Z(\C)$ decomposes into
the sum of 3 simple objects $I(\be)=\be \oplus A\oplus B$; we can assume that 
$\FPdim(B)=\frac{\FPdim(\C)}2$ whence $\FPdim(A)=\frac{\FPdim(\C)}2-1$. Using Proposition 
\ref{FI} we deduce that
\begin{equation}
F(A)=\be \oplus \frac{m}2X,\; F(B)=\be \oplus Y\oplus \frac{m}2X
\end{equation}
In particular, $m$ is forced to be even. Theorem \ref{zero} states that $\theta_A=\theta_B=1$.

Next we consider $I(Y)$. By Proposition \ref{FI} we have $\dim \Hom_{\Z(\C)}(I(Y),I(Y))=3$ and 
$\dim \Hom(I(Y),I(\be))=1$; we also have by definition $\dim \Hom_{\Z(\C)}(I(Y),B)=1$. Thus we
have a decomposition $I(Y)=B+C+D$ where $C$ and $D$ are distinct simple objects different from $A$ and $B$. Using Proposition \ref{FI} we have 
\begin{equation}
F(C)=Y+\alpha X,\; F(D)=Y+\beta X
\end{equation}
where $\alpha, \beta \in \BZ_{\ge 0}$ and $\alpha +\beta=\frac{m}2$. It is clear from 
Theorem \ref{zero} that $\theta_C=\theta_D=-1$.

Finally consider $I(X)$. We have
$$\dim \Hom_{\Z(\C)}(I(X), A)=\dim \Hom_{\Z(\C)}(I(X), A)=\frac{m}2,$$
$$\dim \Hom_{\Z(\C)}(I(X), C)=\alpha ,\; \dim \Hom_{\Z(\C)}(I(X), D)=\beta .$$
Thus 
\begin{equation}
I(X)=\frac{m}2A+\frac{m}2B+\alpha C+\beta D+\sum_i\gamma_i E_i
\end{equation}
where $\gamma_i\in \BZ_{>0}$ and the objects $E_i\in \Z(\C)$ satisfy $F(E_i)=\gamma_iX$.
Using Proposition \ref{FI} we compute
$$\frac{m^2}4+\frac{m^2}4+\alpha^2+\beta^2+\sum_i\gamma_i^2=\dim \Hom(F(I(Y),Y)=m^2+4$$
whence 
\begin{equation} \label{gam}
\sum_i\gamma_i^2=\frac{m^2}2+4-\alpha^2-\beta^2=\frac{m^2}4+4+2\alpha \beta .
\end{equation}
Set $\theta_i=\theta_{E_i}$. Theorem \ref{zero} implies
$$\frac{m}2(1+\frac{m}2d)+\frac{m}2(2+\frac{m}2d)-\alpha (1+\alpha d)-\beta (1+\beta d)+\sum_i\gamma_i^2\theta_id=0$$
whence
\begin{equation} \label{gams}
\sum_i\gamma_i^2\theta_i=\alpha^2+\beta^2-\frac{m^2}2-\frac{m}d=
-2\alpha \beta -\frac{m}2\sqrt{\frac{m^2}4+2}.
\end{equation}

Similarly, Theorem \ref{FS} implies
$$\frac{m}2(1+\frac{m}2d)+\frac{m}2(2+\frac{m}2d)+\alpha (1+\alpha d)+\beta (1+\beta d)+\sum_i\gamma_i^2\theta_i^2d=\pm \dim(\C)=\pm(4+md)$$
whence
\begin{equation} \label{gamfs}
\sum_i\gamma_i^2\theta_i^2=2\alpha \beta -\frac{m^2}4-(m\mp 2)\sqrt{\frac{m^2}4+2}.
\end{equation}

\subsection{Proof of Theorem \ref{near}}
Assume that $m>2$ and the based ring $K(1,0,m,0)$ is categorifiable. Recall that this implies
that $m$ is even. Combining equations 
\eqref{gam} and \eqref{gamfs} we obtain
$$\frac{m^2}4+4+2\alpha \beta =\sum_i\gamma_i^2\ge |\sum_i\gamma_i^2\theta_i^2|=
\frac{m^2}4+(m\mp 2)\sqrt{\frac{m^2}4+2}-2\alpha \beta .$$
Equivalently, $4\alpha \beta +4\ge (m\mp 2)\sqrt{\frac{m^2}4+2}$, which implies
$$\frac{m^2}4+4\ge 4\alpha \beta +4\ge (m\mp 2)\sqrt{\frac{m^2}4+2}
\ge (m-2)\sqrt{\frac{m^2}4+2}.$$
It is easy to see that the inequality $\frac{m^2}4+4\ge (m-2)\sqrt{\frac{m^2}4+2}$ fails for
$m\ge 6$, so we got a contradiction in this case.

In the remaining case $m=4$ we get from \eqref{gam}, \eqref{gams}, \eqref{gamfs}:
$$\sum_i\gamma_i^2=8+2\alpha \beta,\;  \sum_i\gamma_i^2\theta_i=-2\alpha \beta -2\sqrt{6},\;
\sum_i\gamma_i^2\theta_i^2=2\alpha \beta -4-(4\mp 2)\sqrt{6}$$
where $\alpha,\beta \in \BZ_{\ge 0}$, $\alpha +\beta =2$. 
Comparing $|\sum_i\gamma_i^2\theta_i^2|$ with $\sum_i\gamma_i^2$
as above we receive that $\alpha =\beta =1$ and 
$\sum_i\gamma_i^2\theta_i^2=-2-2\sqrt{6}=\sum_i\gamma_i^2\theta_i$, $\sum_i\gamma_i^2=10$.

It follows easily from \cite[Theorem 2]{Lo} that there is a unique way to represent $-2-2\sqrt{6}$ as 
a sum of 10 roots of unity, namely
$$-2-2\sqrt{6}=2(-1)+2\zeta_{24}^7+2\zeta^{-7}_{24}+2\zeta_{24}^{11}+2\zeta^{-11}_{24}$$
where $\zeta_{24}=e^{\pi i/12}$. Since this set of roots of unity is not closed under squaring,
we have a contradiction. Theorem \ref{near} is proved. $\square$

\begin{remark} One can give an alternative proof of Theorem \ref{near} for $m\ge 6$ using only \eqref{gam} and \eqref{gams}. For this one deduces from \cite{Lo} that at least
 $2\alpha \beta +m\sqrt{\frac13\left(\frac{m^2}4+2\right)}$ roots of unity are required in order to represent $-2\alpha \beta -\frac{m}2\sqrt{\frac{m^2}4+2}$ as their sum and obtains a contradiction with \eqref{gam}.
\end{remark}

\subsection{The ring $K(1,0,2,0)$} The fusion categories $\C$ with $K(\C)=K(1,0,2,0)$ were classified in \cite{HH}. We sketch here an alternative argument for this classification.
Recall that $I(\be)\in \Z(\C)$ has a natural structure of separable algebra, see e.g 
\cite[Lemma 3.5]{DMNO}. Moreover, \cite[Theorem 4.10]{DMNO} implies that this algebra has a separable subalgebra $I_0$ with 
$\FPdim(I_0)=\frac12\FPdim(I(\be))$. This forces $I_0=\be \oplus A$. Thus 
$F(I_0)=2\be \oplus X\in \C$ has a structure of separable algebra. This algebra has to be decomposable, so $\be \oplus X\in \C$ has a structure of separable algebra. It is easy to compute the "principal graph" of the category of $\be \oplus X-$modules in $\C$; it is precisely the Dynkin diagram of type $E_6$. Now the results of \cite{HH} follow from the classification of subfactors of type $E_6$, at least in the case when the category $\C$ admits a unitary structure. 

\subsection{Proof of Theorem \ref{uraa}} Let $\C$ be a fusion category of rank 3 admitting a pivotal
structure. Then by Theorems \ref{simple3} and \ref{near} the Grothendieck ring $K(\C)$ is isomorphic 
to one of the following: $K(\BZ/3\BZ)$, $K(1,1,1,0)$, $K(1,0,0,0)$, $K(1,0,1,0)$ or $K(1,0,2,0)$.
For the first 4 rings the result follows from the discussion in \cite[Section 4]{O3}, and for the last ring
the result follows from \cite[Theorem 1]{HH}. \qed

\section{Spherical structures}
\subsection{Sphericalization} \label{sphn}
Let $\C$ be a fusion category. A {\em sphericalization} $\tilde \C$ of $\C$ is defined in \cite[Remark 3.1]{ENO}.
Thus $\tilde \C$ is a spherical fusion category together with a tensor functor $\tilde F: \tilde \C \to \C$. 
The category $\C$ itself has a spherical structure if and only if the functor $\tilde F$ admits a tensor section $\C \to \tilde \C$. Equivalently, $\tilde \C$ should contain a fusion subcategory such that
restriction of $\tilde F$ to this subcategory is an equivalence.

The functor $\tilde F$ has the following properties: it maps simple objects to simple objects and for
any $X\in \O(\C)$ there are precisely two objects $X_+, X_-\in \O(\tilde \C)$ with $\tilde F(X_\pm)=X$
(the choice of $X_+$ and $X_-$ is arbitrary except for $\be_+=\be$). It follows that
$\be_-\ot \be_-=\be, \dim(\be_-)=-1$, and $X_\pm\ot \be_-=\be_-\ot X_\pm=X_\mp$. 
The following observation (\cite[proof of Lemma 2.3]{Th}) is quite useful:

(a) If $X\in \O(\C)$ is self-dual, then both $X_+$ and $X_-$ are self-dual.

Let $\varepsilon =[\be_-]\in K(\tilde \C)$; clearly $\varepsilon$ is a central element. Observe
that $K(\tilde \C)/\langle \varepsilon =1\rangle$ with a basis consisting of images of $[X_+]$ is
isomorphic to $K(\C)$ as a based ring (via the map $[X_+]\mapsto [X]$). Consider the ring
$\tilde K(\C)=K(\tilde \C)/\langle \varepsilon =-1\rangle$. This is a ring over $\BZ$ with a basis
consisting of images of $[X_+]$. In a sense this is just a basis up to signs since the signs of basis
elements depend on the choice of $X_\pm$ above. The following statements are easy to verify:

(b) Let $N$ be a structure constant of $K(\C)$ and let $\tilde N$ be the corresponding structure
constant of $\tilde K(\C)$ (remind that the bases of both rings are labeled by the same set $\O(\C)$).
Then $$|\tilde N|\le N,\;\; \mbox{and}\;\; \tilde N\equiv N\pmod{2}.$$

(c) The category $\C$ has a spherical structure if and only if there is a choice of signs of basis elements
of $\tilde K(\C)$ such that the corresponding structure constants equal.

\subsection{Quadratic rings}
Let us consider a based ring $K(k,l,m,n)$ and let $p(t)$ be the characteristic polynomial of operator $R$,
see Section \ref{notations}. We have the following possibilities for factorization of $p(t)$ over $\BQ$:

(i) $p(t)$ is irreducible. We say that $K(k,l.m,n)$ is {\em cubic} in this case.

(ii) $p(t)$ has irreducible quadratic factor. In this case we say that $K(k,l.m,n)$ is {\em quadratic}.

(iii) $p(t)$ factorizes into three linear factors. We say that $K(k,l,m,n)$ is {\em rational}.


If $K(\C)=K(k,l,m,n)$ is rational, then $\FPdim(\C)$ is integer, so $\C$ is spherical by
\cite[Proposition 8.24]{ENO}. We also have the following

\begin{theorem} \label{quadra}
Let $\C$ be a fusion category such that $K(\C)=K(k,l,m,n)$ is quadratic.
Then $\C$ has a spherical structure.
\end{theorem}

\begin{proof} If $\FPdim(\C)\in \BZ$ then $\C$ has spherical structure by \cite[Proposition 8.24]{ENO}. Thus we 
will assume for the rest of the proof that $\FPdim(\C)\not \in \BZ$. The polynomial $p(t)$ factorizes
$p(t)=(t-\gamma)(t^2-\alpha t+\beta)$ where $\alpha, \beta, \gamma \in \BZ$ and the second factor is irreducible. From the decomposition
$$(t-\gamma)(t^2-\alpha t+\beta)=t^3-(\alpha+\gamma)t^2+(\beta +\alpha \gamma)t-
\gamma \beta$$
we find 
\begin{equation} \label{alphagamma}
b=\alpha +\gamma ,\; c=\beta+\alpha \gamma =\gamma \beta ,\; 
\beta =\alpha \frac{\gamma}{\gamma -1}.
\end{equation}
In particular, we see that $\gamma -1$ is a divisor of $\alpha$ and 
$\beta=\frac{\alpha}{\gamma -1}\gamma \ge \gamma$.

We now consider the sphericalization $\tilde \C$ and the ring $\tilde K(\C)$. 
Using \ref{sphn} (a) and (b) we see that $\tilde K(\C)=K(\tilde k, \tilde l, \tilde m, \tilde n)$ 
(see \eqref{frules}) where
$$|\tilde k|\le k, |\tilde l|\le l, |\tilde m|\le m, |\tilde n|\le n; k-\tilde k, l-\tilde l, m-\tilde m, n-\tilde n\; \; \mbox{are even.}$$
By changing the signs of basis elements of $\tilde K(\C)$ we can (and will) assume that
$$\tilde k+\tilde n\ge 0\; \; \mbox{and}\;\; \tilde l+\tilde m\ge 0.$$
We will generally use tilde for notations associated with $\tilde K(\C)$ parallel to those in $K(\C)$, 
for example $\tilde R, \tilde p(t)$ etc. Notice that the sum of inverse roots of $\tilde p(t)$ equals 1:
indeed the formal codegrees of $K(\tilde \C)$ are precisely doubled formal codegrees of $K(\C)$ and
doubled roots of $\tilde p(t)$, so the result follows from Proposition \ref{trRinv}. Thus
$$\tilde p(t)=t^3-\tilde bt^2+\tilde ct-\tilde c.$$
It follows from \eqref{beq} that $\tilde b\le b$. Moreover, the equality $\tilde b=b$ implies $\tilde k=k,
\tilde l=l, \tilde m=m, \tilde n=n$, so the category $\C$ is spherical by \ref{sphn} (c).
Thus in the rest of this proof we will assume that $\tilde b<b$ and derive a contradiction.

Observe that the dimension homomorphism $K(\tilde \C)\to k$ factors through $\tilde K(\C)$,
so the degree of the dimension field $L(\tilde \C)$ over $\BQ$ is at most three. On the other hand
it follows from Corollary \ref{dfield} that $L(\tilde \C)$ contains the splitting field of $t^2-\alpha t+\beta$;
it follows that $L(\tilde \C)$ is precisely this field. Thus we have two homomorphisms $\tilde K(\C)\to k$
with values in $L(\tilde \C)$: the dimension homomorphism and its Galois conjugate. Thus the remaining
homomorphism must be Galois invariant, so it lands in $\BZ$. In particular, the polynomial $\tilde p(t)$
must be reducible. 

Assume that $\tilde p(t)$ has three integer roots $\tilde f_1, \tilde f_2, \tilde f_3$. 
Two of these roots corresponding to the dimension
homomorphism and its Galois conjugate must be equal, say $\tilde f_1=\tilde f_2$. From the
equation $\frac1{\tilde f_1}+\frac1{\tilde f_2}+\frac1{\tilde f_3}=1$ we find that
$(\tilde f_1, \tilde f_2, \tilde f_3)=(3,3,3)$ or $(\tilde f_1, \tilde f_2, \tilde f_3)=(4,4,2)$.
Hence $\tilde b=9$ or $10$. From $\tilde b=9$ we get a contradiction using \eqref{beq} and
\eqref{klmn}. Thus $\tilde b=10$ and $\dim(\C)=\tilde f_1=4$. Therefore $\gamma$ divides $4$ by
Corollary \ref{fcdivis}. Observe that $\gamma=f_E$ where $E$ is one dimensional representation
of $K(\C)$ corresponding to $\BZ-$valued character $\chi$. Thus $\gamma =1+s^2+t^2$ where
$s=\chi(X)\in \BZ$ and $t=\chi(Y)\in \BZ$, see \cite[Example 2.4]{Od}. It follows that $\gamma=2$
and in view of Remark \ref{bas} (i) we can assume that $s=0$. Then \eqref{frules} implies that $l=0$
and $K(\C)=K(1,0,m,0)$. However in this case $\C$ is near-group category, so it is spherical
by \cite[Theorem 1.3]{Th}. Hence Theorem \ref{quadra} holds in this case.

Thus we will assume that $\tilde p(t)$ is a product of linear factor and of irreducible quadratic factor
$$\tilde p(t)=(t-\tilde \gamma)(t^2-\tilde \alpha t+\tilde \beta).$$
Note that $\dim(\C)$ is one of the roots of $t^2-\tilde \alpha t+\tilde \beta$,
and the splitting fields of $p(t)$ and of $\tilde p(t)$ coincide

Clearly,
\begin{equation}
\tilde b=\tilde \alpha +\tilde \gamma ,\; \tilde c=\tilde \beta+\tilde \alpha \tilde \gamma =\tilde \gamma \tilde \beta ,\; 
\tilde \beta =\tilde \alpha \frac{\tilde \gamma}{\tilde \gamma -1}.
\end{equation}

We have $\tilde \alpha+\tilde \gamma=\tilde b<b=\alpha +\gamma$ and $\frac{\tilde \gamma}{\tilde \gamma -1}\le 2$ whence
$$\tilde \beta=\tilde \alpha \frac{\tilde \gamma}{\tilde \gamma -1}<2(\alpha +\gamma -\tilde \gamma)=2(\beta \frac{\gamma -1}{\gamma}+\gamma -\tilde \gamma)<4\beta$$
(recall that $\beta \ge \gamma$). Since $\frac{\tilde \beta}{\beta}\in \BZ$ by Corollary \ref{fcdivis}, we have the following possibilities:

{\em Case 1:} $\tilde \beta=3\beta$. In this case $2\tilde \alpha\ge \tilde \alpha\frac{\tilde \gamma}{\tilde \gamma-1}=
3\alpha \frac{\gamma}{\gamma-1}$ whence 
$\tilde \alpha\ge \frac32\alpha +\frac32\frac{\alpha}{\gamma -1}$. Notice that by Corollary \ref{fcdivis} 
$\gamma$ divides
$\dim(\C)$, hence $\gamma^2$ divides the norm $\tilde \beta$ of $\dim(\C)$ which is equivalent to
$\frac{3\alpha}{\gamma (\gamma -1)}\in \BZ$. If $\frac{3\alpha}{\gamma (\gamma -1)}\ge 2$ then
$\frac32\frac{\alpha}{\gamma -1}\ge \gamma$, so $\tilde \alpha>\alpha +\gamma$ which contradicts to the inequality $\tilde \alpha+\tilde \gamma<\alpha +\gamma$. If $\frac{3\alpha}{\gamma (\gamma -1)}=1$ then $\alpha =\frac13\gamma(\gamma-1),\; \beta=\frac13\gamma^2$. Thus $\gamma$ is divisible by 3, so $\frac{\alpha^2}{\beta}=\frac13(\gamma-1)^2\not \in \BZ$, whence $\FPdim(\C)$ is not a d-number and we have a contradiction with Proposition \ref{dnumt}. 

{\em Case 2:} $\tilde \beta=2\beta$. In this case $2\tilde \alpha\ge \tilde \alpha\frac{\tilde \gamma}{\tilde \gamma-1}=
2\alpha \frac{\gamma}{\gamma-1}$ whence 
$\tilde \alpha\ge \alpha +\frac{\alpha}{\gamma -1}$. Notice that by Corollary \ref{fcdivis} $\gamma$ divides
$\dim(\C)$, hence $\gamma^2$ divides the norm $\tilde \beta$ of $\dim(\C)$ which is equivalent to
$\frac{2\alpha}{\gamma (\gamma -1)}\in \BZ$. If $\frac{2\alpha}{\gamma (\gamma -1)}\ge 2$ then
$2\frac{\alpha}{\gamma -1}\ge \gamma$, so $\tilde \alpha>\alpha +\gamma$ which contradicts to the inequality $\tilde \alpha+\tilde \gamma<\alpha +\gamma$. If $\frac{2\alpha}{\gamma (\gamma -1)}=1$ then $\alpha =\frac12\gamma(\gamma-1),\; \beta=\frac12\gamma^2$. Thus $\gamma$ is even, so $\frac{\alpha^2}{\beta}=\frac12(\gamma-1)^2\not \in \BZ$, whence $\FPdim(\C)$ is not a d-number and we have a contradiction with Proposition \ref{dnumt}.

{\em Case 3:} $\tilde \beta=\beta$. First of all we claim that $\tilde \gamma<\gamma$. Indeed inequality
$\tilde \gamma\ge \gamma$ implies $\tilde \alpha=\alpha \frac{1+\frac1{\gamma-1}}{1+\frac1{\tilde \gamma-1}}\ge \alpha$, so $\tilde \alpha+\tilde \gamma\ge \alpha +\gamma$.

The splitting fields of the polynomials $t^2-\alpha t+\beta$ and $t^2-\tilde \alpha t+\tilde \beta$ should coincide, equivalently the ratio of their discriminants $\frac{\alpha^2-4\beta}{\tilde \alpha^2-4\tilde \beta}=
\frac{\frac{\alpha^2}{\beta}-4}{\frac{\tilde \alpha^2}{\beta}-4}$ should be a square of rational number (notice that the numerator and denominator of the last expression are integers since $\dim(\C)$ and $\FPdim(\C)$ are d-numbers). Thus there exist a square free integer $d$ and positive integers $x,y$ such that
\begin{equation}
\frac{\alpha^2}{\beta}-4=\left(\frac{\gamma-1}{\gamma}\right)^2\beta-4=dx^2,\;
\frac{\tilde \alpha^2}{\beta}-4=\left(\frac{\tilde \gamma-1}{\tilde \gamma}\right)^2\beta-4=dy^2.
\end{equation}
Henceforth we have
$$\frac{dx^2+4}{dy^2+4}=\left(\frac{1+\frac1{\tilde \gamma-1}}{1+\frac1{\gamma-1}}\right)^2.$$
Let $\frac{1+\frac1{\tilde \gamma-1}}{1+\frac1{\gamma-1}}=\frac{A}B$ where $\frac{A}B$ is in its lowest terms. Notice that inequalities $2\le \tilde \gamma<\gamma$ imply $1<\frac{A}B<2$. The equality
$\frac{dx^2+4}{dy^2+4}=\left(\frac{A}B\right)^2$ is equivalent to
$$d(Bx-Ay)(Bx+Ay)=4(A^2-B^2).$$
Since $d$ is square free, the numbers $Bx-Ay$ and $Bx+Ay$ are both even. We set 
$q=\frac{Bx-Ay}2$, so $\frac{Bx+Ay}2=\frac{A^2-B^2}{dq}$. Hence
\begin{equation}\label{AyBx}
Bx=\frac{A^2-B^2}{dq}+q=\frac{A^2-B^2+dq^2}{dq},\; Ay=\frac{A^2-B^2}{dq}-q=\frac{A^2-B^2-dq^2}{dq}.
\end{equation}
These equalities imply that $A^2+B^2+dq^2$ is divisible by both $A$ and $B$, so it is divisible
by $AB$. We claim that 
$$2<\frac{A^2+B^2+dq^2}{AB}<4.$$
Indeed, from \eqref{AyBx} we get $A^2-B^2-dq^2=Aydq>0$, so 
$\frac{A}B-\frac{B}A>\frac{dq^2}{AB}$, whence $\frac{dq^2}{AB}<\frac32$ (since $\frac32$ is the
maximal value of function $x-\frac1x$ for $x\in [1,2]$). Now $\frac{A^2+B^2+dq^2}{AB}=\frac{A}B+\frac{B}A+\frac{dq^2}{AB}$ where $\frac{A}B+\frac{B}A\in (2,\frac52)$ (since $\frac{A}B\in (1,2)$) and the desired inequality is established. The previous remarks imply
$$\frac{A^2+B^2+dq^2}{AB}=3$$
and, eliminating $dq^2$ we get
\begin{equation}\label{xy}
x=\frac{3A-2B}{dq},\; y=\frac{2A-3B}{dq}.
\end{equation}
Thus $dq$ is divisor of $5A=3(3A-2B)-2(2A-3B)$ and $5B=2(3A-2B)-3(2A-3B)$; since $A$ and
$B$ are coprime we see that $dq$ divides $5$. It follows that $dq^2=1,5$ or $25$.

{\em Subcase $dq^2=25$.} The equality $A^2+B^2+25=3AB$ is equivalent to
$(2A-3B)^2-5B^2+100=0$ which implies that both $A$ and $B$ are divisible by 5. This is a contradiction.

{\em Subcase $dq^2=1$.} The equality $A^2+B^2+1=3AB$ is equivalent to $\frac1{B^2}=3\frac{A}B-\frac{A^2}{B^2}-1$ which implies $\frac1{B^2}>1$ (since the minimal value of the function $3x-x^2-1$ on the interval $[1,2]$ is 1), which is a contradiction.

{\em Subcase $dq^2=5$.} As in the previous case we get $\frac5{B^2}>1$ whence $\frac{A}B=\frac32$. Then \eqref{xy} gives $y=0$ and we get a final contradiction.
\end{proof}
\subsection{Examples for cubic rings}
It is easy to see that the ring $K(k,l,m,n)$ is cubic if and only if $K(k,l,m,n)\ot_\BZ \BQ$ is a field.
The following result is useful in order to determine whether cubic $K(k,l,m,n)$ passes
the cyclotomic test:
\begin{proposition} \label{discc}
 Assume that $K(k,l,m,n)\ot_\BZ \BQ$ is a field. This field is abelian
if and only if $c$ is a square of integer.
\end{proposition}

\begin{proof} Observe that the matrix of the trace form on the ring $K(k,l,m,n)$ with respect to its
basis coincides with the matrix \eqref{matrix}. Thus the discriminant of the field 
$K(k,l,m,n)\ot_\BZ \BQ$ is $c$ modulo rational squares. Finally it is well known that a cubic field is abelian if and only if its discriminant is square.
\end{proof}

\begin{remark}\label{ramic}
 The proof of Proposition \ref{discc} shows that the primes ramified in 
the field $K(k,l,m,n)\ot_\BZ \BQ$ are divisors of $c$. Of course it is possible that some 
prime divisor of $c$ does not ramify. However if $K(k,l,m,n)$ passes the d-number test (see 
Proposition \ref{dnumt})
then a root of $p(t)=t^3-bt^2+ct-c$ is $\sqrt[3]{c}$ modulo units, so if $c=\prod_ip_i^{\alpha_i}$ 
(with distinct primes $p_i$) and $\alpha_i$ is not divisible by 3 then $p_i$ is ramified in
$K(k,l,m,n)\ot_\BZ \BQ$.
\end{remark}

The cubic based rings $K(k,l,m,n)$ which pass both the cyclotomic and d-number tests from Section \ref{prefus}
are very sparse (but probably there are infinitely many of them). 
A complete list of such rings with $l\le k\le 100000$ obtained as a result of computer search
includes categorifiable ring $K(1,1,1,0)$ and the following rings:


\begin{equation*}
\begin{array}{|c||c||c|}
\hline
k,l,m,n&k,l,m,n&k,l,m,n\\
\hline
29,13,62,7&1493,863,2529,530&34487,23653,50489,16090\\
\hline
83,77,91,70&2339,323,16906,49&43037,25271,74851,13924\\
\hline
305,179,530,99&7579,4063,14136,2179&47603,29251,77448,17987\\
\hline
409,331,137,566&8401,5099,13874,3075&54559,41609,71568,31711\\
\hline
1133,169,7624,21&8621,473,157182,23&86593,16571,453173,3042\\
\hline
1373,31,60753,2&20341,3887,106288,773&95705,14221,641435,2506\\
\hline
\end{array}
\end{equation*}

\begin{remark} \label{codd}
The following observations (especially the second one) are useful in the above 
mentioned computer search:

(i) If $b^3$ is divisible by $c$ then $p(t)$ is irreducible.
 Indeed, assume that $p(t)$ is reducible, $p(t)=(t-\gamma)(t^2-\alpha t+\beta)$. 
 Then \eqref{alphagamma} implies that $s:=\frac{\alpha}{\gamma -1}$ is an integer and we have
 $b=\alpha +\gamma =s\gamma +\gamma -s$, $c=\gamma \beta=\gamma^2s$. Hence $\frac{b^3}c\in \BZ$ is equivalent to $3\frac{s^2}{\gamma}+\frac{\gamma}s+3\frac{s}{\gamma}-\frac{s^2}{\gamma^2}\in \BZ$. This condition implies that $\gamma$ and $s$ must have the same $p-$adic
 valuation for each prime $p$, whence $\gamma =s$. Thus $b=s^2$ which is impossible since 
 by \eqref{beq} $b$ is congruent to 2 or 3 modulo 4.

(ii) $k$ and $l$ should be odd (which implies that $c$ is odd and $b\equiv 2\pmod{4}$).
Indeed if $k$ or $l$ is even then \eqref{klmn} and \eqref{magic} imply that $c$ is even, hence
$b$ is even by the d-number test. Then \eqref{beq} shows that $b\equiv 2\pmod{4}$.
Using Proposition \ref{discc} and the d-number test we see that $c$ is divisible by 4 but not by 8.
Hence by Remark \ref{ramic} the field $K(k,l,m,n)\ot_\BZ \BQ$ is ramified at 2. This is
a contradiction since the only extension of 2-adic numbers $\BQ_2$ with Galois group 
$\BZ/3\BZ$ is unramified.
\end{remark}

We describe now the computer computations which show that a categorification of the rings above
is spherical (and hence does not exist). 

1) First of all we can give an estimate of $\tilde c$. Using \eqref{magic} and \eqref{klmn}  we see that
$$\tilde c=4\tilde b-9+(\tilde k\tilde l-\tilde m\tilde n)^2+(2(\tilde k^2-\tilde l\tilde m)-1)^2=4\tilde b-9+(\tilde k\tilde l-\tilde m\tilde n)^2+(2(\tilde l^2-\tilde k\tilde n)-1)^2.$$
Hence,
$$\tilde c\le min(4b-9+(kl+mn)^2+(2k^2+2lm+1)^2,4b-9+(kl+mn)^2+(2l^2+2kn+1)^2).$$

\begin{example} For the ring $K(54559,41609,71568,31711)$ we obtain
$$\frac{\tilde c}c\le 11132295.1<15^6,$$
and for the ring $K(95705,14221,641435,2506)$ we obtain
$$\frac{\tilde c}c<144.6<13^2<3^6.$$
\end{example}

2) Recall that by Corollary \ref{fcdivis} $\tilde c$ should be divisible by $c$; also the fields 
$K(k,l,m,n)\ot_\BZ \BQ \simeq\BQ[t]/(p(t))$ and 
$K(\tilde k,\tilde l,\tilde m,\tilde n)\ot_\BZ \BQ \simeq \BQ[t]/(\tilde p(t))$ 
should be isomorphic (since by Corollary \ref{dfield} $K(k,l,m,n)\ot_\BZ \BQ$ should
be contained in the dimension field $L(\tilde \C)$ which is contained in
$K(\tilde k,\tilde l,\tilde m,\tilde n)\ot_\BZ \BQ$). In view of
Remark \ref{ramic} this shows that the prime factors of $\tilde c$ which do not appear in the prime decomposition of $c$ should have exponents divisible by 6. Using the estimate for $\tilde c$ above
we can get a list of possible values of $\tilde c$ (note that in view of Remark \ref{codd} (ii), $\tilde c$ should be odd).

\begin{example} For the ring $K(54559,41609,71568,31711)$ we have $c=7^2\cdot 127^2\cdot 2791^2$. Thus the possible values of $\tilde c/c$ are: $1, 3^6, 5^6, 7^6, 9^6, 11^6, 13^6, 7^2,
7^2\cdot 3^6, 7^2\cdot 5^6, 7^4, 7^4\cdot 3^6, 7^8, 7^2\cdot 127^2, 127^2, 2791^2.$ 

For the ring $K(95705,14221,641435,2506)$ we have $c=3^6\cdot 13^2\cdot 151^2\cdot 4861^2$
and the possible values of $\tilde c/c$ are $1,3^2,3^4$.
\end{example}


3) For each value of $\tilde c$ the possible values of $\tilde b$ satisfy the following: $\tilde b^3$ is divisible by
$\tilde c$ (and hence by $c$), $0<\tilde b\le b$, and $\tilde b\equiv b\pmod{4}$; this can be enumerated efficiently by computer. 

\begin{example} For the ring $K(54559,41609,71568,31711)$ we get 
$$\tilde b=b-i\cdot 4\cdot 7\cdot 127\cdot 2791,\; i\in \BZ, \; 0\le i\le 2040,$$
and for the ring $K(95705,14221,641435,2506)$ we get
$$\tilde b=b-i\cdot 4\cdot 3^2\cdot 13\cdot 151\cdot 4861,\; i\in \BZ, \; 0\le i\le 1279.$$
\end{example}

4) Then for each $\tilde c$ and $\tilde b$ we can compute the discriminant of $\tilde p(t)$ and
discard all the cases where the field $\BQ[t]/(\tilde p(t))$ is not abelian.  We found that the only
possibilities that are not discarded are either $\tilde b=b, \tilde c=c$ (which implies pseudo-unitarity of a categorification of $K(k,l,m,n)$) or the following two cases: $K(29,13,62,7)$ with 
$\tilde b=378, \tilde c=c=35721$,
and $K(83,77,91,70)$ with $\tilde b=1026, \tilde c=c=263169$.

5) Finally we deal with the two exceptional cases above. 

(a) In the case of $K(29,13,62,7)$ there exists
no ring $K(\tilde k,\tilde l,\tilde m,\tilde n)$ consistent with $\tilde b=378$ and $\tilde c=35721$ (we can find $\tilde k,\tilde l,\tilde m,\tilde n$ from equations $(\tilde k+\tilde n)^2+(\tilde l+\tilde m)^2=\tilde b-9$ and $(3\tilde k-\tilde n)^2+(3\tilde l-\tilde m)^2=\tilde b-1$ which follow from \eqref{beq}
and \eqref{klmn}; up to signs the only solution is $K(7,1,14,5)$ which is incompatible with the value
of $\tilde c$). Note that in this case the fields $\BQ[t]/(p(t))$ and $\BQ[t]/(\tilde p(t))$ are isomorphic
(one isomorphism sends $t\in \BQ[t]/(p(t))$ to $-\frac{13}9t^2+\frac{881}3t-294 \in \BQ[t]/(\tilde p(t))$).

(b) In the case of $K(83,77,91,70)$ the fields $\BQ[t]/(p(t))$ and $\BQ[t]/(\tilde p(t))$ are not isomorphic (here is an easy way to verify this: let $t_1, t_2, t_3$ be roots of $p(t)$ and let
$\tilde t_1, \tilde t_2, \tilde t_3$ be the roots of $\tilde p(t)$. If the fields above are isomorphic then
one of the numbers $t_1\tilde t_1+t_2\tilde t_2+t_3\tilde t_3$ or $t_1\tilde t_1+t_2\tilde t_3+t_3\tilde t_2$ would be integer. However a numerical computation shows that this is not the case.). Note that in this case the ring $K(1,1,23,-22)$ is compatible with the values of $\tilde b$ and $\tilde c$.

We summarize the results of this Section as follows:

\begin{theorem} \label{bigfus}
Assume that a based ring $K(k,l,m,n)$ with $k\ge l$ admits a non-spherical categorification. Then $k>100,000$.
\end{theorem}
 
\appendix
\section{Some results on near-group categories}\label{A}
\centerline{Dmitri Nikshych and Victor Ostrik}
\bigskip
\subsection{Near-group categories}
An easiest class of fusion categories is formed by {\em pointed fusion categories}, that is fusion
categories where all
simple objects are invertible, see e.g. \cite[p. 629]{ENO}. The pointed fusion categories are quite well understood. Thus it is a natural to look for fusion categories which would represent the next level 
of complexity. One such class, the {\em near-group categories} was introduced 
by J.~Siehler in \cite{Sie}. Namely,
a semisimple rigid tensor category $\C$ is called {\em near-group category} if all its simple objects 
except for one are invertible. Let $G$ be the group of isomorphism classes of invertible objects 
in $\C$ and let $X$ be the isomorphism class of the remaining non-invertible simple object. 
Then the group $G$ must be finite (so $\C$ is a fusion category) and there exists 
$n\in \BZ_{\ge 0}$ such that the Grothendieck 
ring $K(\C)$ has the following multiplication:

$$g\cdot h=gh\; \mbox{for}\; g,h\in G; g\cdot X=X\cdot g=X\; \mbox{for}\; g\in G; X^2=\sum_{g\in G}g+nX.$$

We will denote this based ring $K(G,n)$. Thus a near-group category is a categorification
of some $K(G,n)$. We have the following examples:

\begin{example} \label{s4ex}
(i) Let $\bF_q$ be a finite field. The category of representations of the semi-direct product
$\bF_q\rtimes \bF_q^\times$ is a categorification of $K(\BZ/(q-1)\BZ,q-2)$.

(ii) The Tambara-Yamagami categories \cite{TY} are categorifications of $K(G,0)$.

(iii) Many categorifications of $K(G,n)$ where $G$ is abelian and $n=|G|$ were constructed 
in \cite{Iz} and \cite{EG}. These results suggest strongly that there is an infinite series of
categories of this type.

(iv) (cf \cite[Theorem 14.40,II]{IK}) 
Let $\tilde G=S_4$ and let $\tilde H$ be a subgroup of order 4 in $\tilde G$ which is not
normal. Then the group theoretical category $\C(\tilde G,\tilde H,1,1)$ (see \cite[Section 8.8]{ENO})
is a categorification of $K(G,2)$ where $G$ is the dihedral group of order 8. This example is
in apparent contradiction with \cite[Theorem 1.1]{Sie}. This is because the group $G$ is implicitly assumed
to be abelian in {\em loc. cit.} as it was observed in \cite{EG}.
\end{example}

It was proved in \cite[Theorem 1.3]{Th} that any near-group category is a Galois conjugate of pseudo-unitary one. Thus Theorem \ref{I1}
applies in this case. The goal of this Section is to collect some interesting consequences 
of these results. Thus we show the following:


1) Assume that $\C$ is not integral (that is the roots of quadratic equation $d^2=|G|+nd$ are not integers). Then $G$ is abelian and $\frac{n}{|G|}\in \BZ$, see Theorem \ref{A1} 
(cf \cite[Theorem 2(a)]{EG}).

2) Assume that $\C$ is integral. Then $\C$ is weakly group theoretical in the sense of \cite{ENOw},
see Theorem \ref{A2}.


3) Assume that $G$ is abelian and $0<n<|G|$. Then $n=|G|-1$ and $G$ is cyclic (thus all such
categories are classified in \cite[Corollary 7.4]{EGO}), see Theorem \ref{A3} (cf \cite[Theorem 1.2]{Sie}). 
Note that Example \ref{s4ex} (iv) shows that
the assumption that $G$ is abelian can not be omitted.

\begin{remark} When this paper was completed we learnt that Izumi \cite{Izz} obtained very interesting
results on near-group $C^*-$categories. In particular he obtained counterparts of Theorems \ref{A1} and
\ref{A3}.
\end{remark}


\subsection{Irreducible representations of $K(G,n)$}
In what follows we will fix an embedding of the subfield of algebraic integers in $k$ to $\BC$
such that the near-group category in question is pseudo-unitary (such embedding 
exists by \cite[Theorem 1.3]{Th}) and we will endow the category with a pseudo-unitary spherical structure.

Let $\Irr_0(G)$ be the set of non-trivial irreducible characters of the group $G$.
Let $\rho_\chi: G\to GL(E)$ be a representation of $G$ over $k$ affording character $\chi \in \Irr_0(G)$.
Then $X\mapsto 0, g\mapsto \rho_\chi (g)$ is a well defined homomorphism $\tilde \rho_\chi: K(G,n)\to \End(E)$.
It is clear that $\tilde \rho_\chi$ is an irreducible representation of $K(G,n)$ over $k$. We will denote by $E_\chi$ 
the corresponding $K(G,n)-$module.

Let $d_+$ and $d_-$ be the positive and the negative roots of the equation $d^2=|G|+nd$. Then
the homomorphisms $\rho_\pm : K(G,n)\to k$ with $\rho_\pm (g)=1$ and $\rho_\pm (X)=d_\pm$ are well defined 
(note that $\rho_+$ coincides with the Frobenius-Perron dimension homomorphism). We will denote by $E_\pm$ the
corresponding one dimensional $K(G,n)-$modules.

\begin{lemma} \label{ngirr}
 Any irreducible $K(G,n)-$module over $k$ is isomorphic to either $E_\chi$ for some
$\chi \in \Irr_0(G)$, or to $E_\pm$.
\end{lemma}

\begin{proof} We have
$$\sum_{\chi \in \Irr_0(G)}\dim(E_\chi)^2+\dim(E_+)^2+\dim(E_-)^2=2+\sum_{\chi \in \Irr_0(G)}\chi(1)^2=|G|+1.$$
Since the rank of $K(G,n)$ over $\BZ$ is $|G|+1$ the result follows.
\end{proof}

\begin{lemma} \label{ngfc}
 The formal codegrees of irreducible representations of $K(G,n)$ are
$f_{E_\chi}=\frac{|G|}{\chi(1)}$ for $\chi \in \Irr_0(G)$, and $f_{E_\pm}=2|G|+nd_\pm$.
\end{lemma}

\begin{proof} We have $\sum_ib_ib_{i^*}=|G|+X^2$, so the result follows from \cite[Lemma 2.6]{Od}.
\end{proof}

Assume that $\C(G,n)$ is a near-group category with $K(\C(G,n))=K(G,n)$.
The following result follows immediately from Theorem \ref{I1} and Lemmas \ref{ngirr} and \ref{ngfc}.

\begin{proposition} \label{ngi1}
We have the following decomposition of $I(\be)\in \Z(\C(G,n))$ into the sum of irreducible objects:
$$I(\be)=\be +A_{E_-}+\sum_{\chi \in \Irr_0(G)}\chi(1)A_{E_\chi}$$
where $\dim(A_{E_-})=1+\frac{n}{|G|}d_+$ and $\dim(A_{E_\chi})=\chi(1)(2+\frac{n}{|G|}d_+)$.
\end{proposition} \qed
 
 We can now prove the first main result of this Section:
 
 \begin{theorem} \label{A1}
 Assume that the ring $K(G,n)$ is categorifiable and $\dim(X)=d_+$ is irrational. Then 

(i) $n$ is divisible by $|G|$;

(ii) the group $G$ is abelian.
 \end{theorem}
 
\begin{remark} In the setting of $C^*-$categories it was proved by Evans and Gannon (\cite{EG})
that (i) holds if $G$ is abelian.
\end{remark}

 \begin{proof} (i) Assume that $F(A_{E_-})$ is a sum of $\alpha$ invertible objects and $\beta$ 
 copies of $X$
(thus $\alpha, \beta \in \BZ_{\ge 0}$). Then $\dim(A_{E_-})=\alpha +\beta d_+$. On the other hand by Proposition \ref{ngi1}
$\dim(A_{E_-})=1+\frac{n}{|G|}d_+$. Since $d_+$ is irrational this implies that $\beta =\frac{n}{|G|}$ and the result follows.

(ii) Assume that the conclusion fails. Then there exists $\chi \in \Irr_0(G)$ with
 $\chi(1)>1$. Then $\dim(A_{E_\chi})=2\chi(1)+\frac{n\chi(1)}{|G|}d_+$. By the argument in (i) we get that
$[F(A_{E_\chi}):X]=\frac{n\chi(1)}{|G|}$. Also by Theorem \ref{I1} we have $[F(A_{E_\chi}):\be]=\chi(1)$. 
Since $[F(A_{E_\chi}):X]\dim(X)+[F(A_{E_\chi}):\be]\dim(\be)=\chi(1)+\frac{n\chi(1)}{|G|}d_+<\dim(A_{E_\chi})$,
we see that $F(A_{E_\chi})$ contains at least one nontrivial invertible object. Thus this object appears in $F(I(\be))$
at least $\chi(1)>1$ times. But this is a contradiction since by Proposition \ref{FI} any nontrivial invertible object
appears in $F(I(\be))$ exactly once.
 \end{proof}

\begin{remark} The argument in (ii) sometimes applies even if $d_+\in \BQ$, for example in the case $n=0$. Thus
we get yet another proof that for a Tamabara-Yamagami category $\C(G,0)$ the group $G$ must be abelian, cf \cite[Corollary 3.3]{TY} and \cite[Proposition 9.3 (i)]{ENOh}.
\end{remark}

\subsection{Integral case} 
\begin{lemma} \label{nginin}
Assume that $\C(G,n)$ is integral. Then $n<|G|$.
\end{lemma}

\begin{proof}
We have $|G|=d_+(d_+-n)$ where $d_+\in \BZ_{\ge 0}$. Hence 
$$n=d_+-(d_+-n)<d_+\le |G|.$$
\end{proof}

\begin{corollary} \label{ngin}
Assume that $\C(G,n)$ is integral. Then $[F(A_{E_-}):X]=0$. 
\end{corollary}

\begin{proof} By Proposition \ref{ngi1} and Lemma \ref{nginin} we have 
$\dim(A_{E_-})=1+\frac{n}{|G|}d_+<
1+d_+$. The result follows.
\end{proof}

\begin{theorem} \label{A2}
Assume that $\C(G,n)$ is integral. Then $\C(G,n)$ is weakly group theoretical.
\end{theorem}

\begin{proof}
The Tambara-Yamagami categories $\C(G,0)$ are $\Z/2\BZ-$extensions of pointed categories,
so they are weakly group theoretical by \cite[Definition 1.1]{ENOw}. Thus we will assume that $n>0$
in the rest of this proof. In particular we have
$$\dim(\C(G,n))=2|G|+nd_+^2>2|G|=2\dim(\Vec_G).$$

Let us consider full subcategory $\D$ of $\Z(\C(G,n))$ consisting of objects  $A$ such that
$F(A)\in \C(G,n)$ is a sum of invertible objects (equivalently $[F(A):X]=0$).
It is clear that $\D$ is a fusion subcategory of $\Z(\C(G,n))$. Corollary \ref{ngin} says that $A_{E_-}\in \D$.
In particular, the restriction of the forgetful functor $F$ to $\D$ is not {\em injective} since 
$F(A_{E_-})$ is not simple. Let $\D'$ be the {\em M\"uger centralizer} of $\D$ in $\Z(\C(G,n))$,
see \cite[Definition 2.6]{Mm}.
It follows from \cite[Theorem 3.12]{DNO} that the restriction of $F$ to $\D'$ is not {\em surjective}. 
Since the object $X$
generates $\C(G,n)$ this implies that for any $A\in \D'$ we have $[F(A):X]=0$. In other words
$\D'\subset \D$ or, equivalently, the category $\D'$ is symmetric.

\begin{lemma} The symmetric category $\D'$ is Tannakian.
\end{lemma}

\begin{proof} For the sake of contradiction assume that $\D'$ is not Tannakian. Then by
\cite[Corollary 2.50]{DGNO} $\D'$ contains a fusion subcategory $\tilde \D'$ with $\dim(\tilde \D')=\frac12\dim(\D')$.
Let $\tilde \D$ be the M\"uger centralizer of $\tilde \D'$ in $\Z(\C(G,n))$, see \cite[Definition 2.6]{Mm}.
Then $\tilde \D \supset \D$ and $\dim(\tilde \D)=2\dim(\D)$, see \cite[Theorem 3.14]{DGNO}. 

By definition of category $\D$ there
exists an object $A\in \tilde \D$ with $[F(A):X]\ne 0$; hence the forgetful functor $F: \tilde \D \to
\C(G,n)$ is surjective. Similarly to \cite[proof of Corollary 8.11]{ENO} we obtain
$$\dim(\tilde \D)=\dim(\C(G,n))\left( \sum_{A\in \O(\tilde \D)}\dim(A)[F(A):\be]\right)>$$
$$2\dim(\Vec_G)\left( \sum_{A\in \O(\D)}\dim(A)[F(A):\be]\right)\ge 2\dim(\D).$$
This is a desired contradiction.
\end{proof}

Let us consider the {\em fiber category} $\D\bt_{\D'}\Vec$ associated with Tannakian subcategory
$\D'\subset \Z(\C(G,n))$.
It follows from \cite[Example 3.14 and Proposition 5.4]{DMNO} that there exists a fusion
category $\A$ and braided equivalence $\D\bt_{\D'}\Vec \simeq \Z(\A)$. 

\begin{lemma} \label{Agt}
The category $\A$ is group theoretical.
\end{lemma}

\begin{proof} The restriction of the forgetful functor $F:  \Z(\C(G,n))\to  \C(G,n)$ to the subcategory
$\D$ gives a central functor (see \cite[Definition 2.4]{DMNO}) $\D \to \Vec_G\subset \C(G,n)$.
Thus there exists a braided functor $\D \to \Z(\Vec_G)$. Let $\D_1$ denote the image of this
functor. Then $\D_1$ is group theoretical category as a subcategory of group theoretical
category $\Z(\Vec_G)$, see \cite[Proposition 8.44 (i)]{ENO}. On the other hand it follows from
\cite[Corollary 3.24]{DMNO} that the canonical functor $\D \to \D\bt_{\D'}\Vec$ factors through
$\D\to \D_1$. Thus the category $\D\bt_{\D'}\Vec =\Z(\A)$ is group theoretical by
\cite[Proposition 8.44 (ii)]{ENO}. Since the forgetful functor $\Z(\A)\to \A$ is surjective 
(see e.g. \cite[Proposition 5.4]{ENO}) we  conclude that $\A$ is group theoretical 
by \cite[Proposition 8.44 (ii)]{ENO}.
\end{proof}

It follows from Lemma \ref{Agt} that the fiber category $\D\bt_{\D'}\Vec$ contains 
a Lagrangian subcategory, see \cite[Theorem 4.64]{DGNO}.
Hence Theorem holds by \cite[Proposition 4.2]{ENOw}.
\end{proof}

It would be interesting to give a classification of integral near-group categories in group theoretical terms. In particular it is not known to the authors whether examples of non group theoretical
near-group categories $\C(G,n)$ with $n>0$ exist (it is known that such examples exist when $n=0$,
see \cite[Remark 8.48]{ENO}).

\begin{theorem} \label{A3}
Let $\C(G,n)$ be a near-group category

{\em (i) (cf \cite[Theorem 1.1]{Sie})} 
Assume that $0<n<|G|-1$. Then $G$ is not abelian.

{\em (ii) (cf \cite[Theorem 1.2]{Sie})} Assume that $n=|G|-1$. Then $G$ is cyclic.
\end{theorem}

\begin{proof}
(i) It is elementary to show that the assumption implies $1+\frac{n}{|G|}d_+<d_+$.
Hence for any $\chi \in \Irr_0(G)$ with $\chi(1)=1$ we have $[F(A_{E_\chi}):X]=0$. Recall
that by Corollary \ref{ngin} $[F(A_{E_-}):X]=0$. Thus there exists $\chi \in \Irr_0(G)$ with $\chi(1)>1$,
so the group $G$ is non abelian.

(ii) In this case $\dim (A_{E_-})=|G|$, so $[F(A_{E_-})]=\sum_{g\in G}g$.
Hence the forgetful functor $\D \to \Vec_G$ is surjective (see proof of Theorem \ref{A2} for the 
definition of $\D$). Using \cite[proof of Corollary 8.11]{ENO}
we find that $\dim(\D)=|G|(|G|+1)=\dim(\C(G,n))$. Hence by \cite[Theorem 3.2]{Mm} we have $\D'=\D$. Thus $\D$
is Lagrangian subcategory of $\Z(\C(G,n))$ and the category $\C(G,n)$ is group theoretical, see 
\cite[Theorem 4.64]{DGNO} and \cite[Theorem 3.1]{ENOw}. Let $\tilde G$ be a finite group
such that we have a braided equivalence $\D \simeq \Rep(\tilde G)$. Observe that the space 
of homomorphisms between
$I(\be)$ and the regular algebra (see \cite[Example 2.8]{DMNO}) of $\D$ has dimension $|G|+1$. 
Thus by Remark \ref{funMN} there is a module category providing weak Morita equivalence between
$\C(G,n)$ and $\Vec_{\tilde G}^\omega$ with $|G|+1$ simple objects. Equivalently,
$\C(G,n)\simeq \C(\tilde G, \tilde H, \omega, \psi)$ (see \cite[Section 8.8]{ENO} for
the notations) where $|\tilde G|=|G|(|G|+1)$ and $|\tilde H|=|G|$. Since $\Rep (\tilde H)$ is 
subcategory of $\C(\tilde G, \tilde H, \omega, \psi)$ we see that $\tilde H$ is abelian and 
$\tilde G$ is a union of exactly 2 double cosets with respect to $\tilde H$. It follows that 
$\tilde G$ is a Frobenius group with Frobenius complement $\tilde H$. Thus arguing as in
\cite[proof of Corollary 7.4]{EGO} we see that $\C(G,n)$ is one of the categories described 
in \cite[Corollary 7.4]{EGO}. In particular,
the group $G$ must be cyclic in this case.
\end{proof}

\bibliographystyle{ams-alpha}

\end{document}